\documentclass[12pt,oneside]{amsart}
\usepackage{graphicx} 
\usepackage[utf8]{inputenc} 
\usepackage{amsmath, amssymb, amsthm, mathtools}
\usepackage{amsfonts} 
\usepackage{amsthm} 
\usepackage{natbib}
\usepackage{kotex}

\newtheorem*{remark*}{Remark} 
\usepackage{tikz}
\usepackage{caption}
\usepackage{geometry}
\usepackage{setspace}
\title[Cutoff for Inhomogeneous Nonlinear Recombination]{Cutoff Phenomenon for Inhomogeneous Nonlinear Recombination in Arbitrary Finite Product Spaces}
\author{Junho Kim, Insuk Seo}
\date{October 2, 2025}

\address{Insuk Seo: Department of Mathematical Sciences and Research Institute of Mathematics, Seoul National University, Republic of Korea. \\
e-mail: \texttt{insuk.seo@snu.ac.kr} }  

\address{Junho Kim: Department of Mathematical Sciences, Seoul National University, Republic of Korea. \\
e-mail: \texttt{jhkim29@snu.ac.kr} }  

\newtheorem{theorem}{Theorem}[section]
\newtheorem{lemma}[theorem]{Lemma}
\newtheorem{proposition}[theorem]{Proposition}
\newtheorem{corollary}[theorem]{Corollary}
\theoremstyle{definition}
\newtheorem{definition}[theorem]{Definition}
\theoremstyle{remark}
\newtheorem{remark}[theorem]{Remark}

\newcommand{\R}{\mathbb{R}}
\newcommand{\N}{\mathbb{N}}
\newcommand{\E}{\mathbb{E}}
\newcommand{\p}{\mathbb{P}}

\newcommand{\Prob}{\mathbb{P}}
\newcommand{\ind}{\mathbf{1}}
\newcommand{\TV}{\mathrm{TV}}
\newcommand{\qvec}{\mathbf{q}}
\newcommand{\fvec}{\mathbf{f}}

\newcommand{\Gvec}{\mathbf{G}}

\begin{document}
\maketitle

\begin{abstract}
In this article, we prove the cutoff phenomenon for a general class of the discrete-time nonlinear recombination models. This system models the evolution of a probability measure on a finite product space $S^n$ representing the state of spins on $n$ sites. Although its stationary distribution has a product structure, and its evolution is Markovian, the dynamics of the model is nonlinear. Consequently, the estimation of the \textit{mixing time} becomes a highly non-trivial task. The special case with two spins and homogeneous stationary measure was considered in Caputo, Labbé, and Lacoin [The Annals of Applied Probability 35:1164–1197, 2025], where the cutoff phenomenon for the mixing behavior has been verified. In this article, we extend this result to the general case with finite spins and inhomogeneous stationary measure by developing a novel algebraic representation for the density fluctuation of the system with respect to its stationary state.  
\end{abstract}

\section{Introduction}

In contrast to the well-developed mixing theory for linear Markov chains (e.g., \cite{LevinPeres2017} for a comprehensive introduction), the corresponding theory for the nonlinear Markovian dynamical systems, where the transition operator depends on the current distribution of the system, presents a significant challenge. The study of such convergence properties is very rare, e.g., \cite{Andrieu2011, Rabinovich1992}. In particular, mixing time analysis quantifying the rate of convergence to stationarity, a central and often challenging problem in Markovian dynamics, remains a nascent and demanding field for nonlinear Markov systems. We remark \cite{ CaputoSinclair2018, CaputoSinclair2023, Rabani1998} for possible references. 

A canonical example in this challenging domain is the nonlinear recombination model, which has its origins in the Hardy-Weinberg principle of population genetics~\cite{Geiringer1944, Hardy1908, Weinberg1908}. While the mixing time for this model was known to be of order $\Theta(\log n)$ from the work of Rabani, Rabinovich, and Sinclair~\cite{Rabani1998}, the cutoff phenomenon was first established in the seminal work of Caputo, Labbé, and Lacoin for the homogeneous two-spin system~\cite{CaputoLabbeLacoin2025}. They established a cutoff phenomenon, locating the sharp transition at time $\log_2 n$ in the discrete-time setting, and at $2 \log n$ in the continuous-time setting, both with an $O(1)$ window. A key aspect of their analysis was the derivation of an explicit convergence profile for monochromatic initial distributions, which enabled them to prove the asymptotic sharpness of their cutoff bounds.

However, the generalization of these results to systems with inhomogeneous marginals and on general product spaces remained unresolved, as noted in~\cite{CaputoLabbeLacoin2025}. The key obstacles were an algebraic framework fundamentally tied to the binary structure; the lack of a clear analogue to the monochromatic distribution for establishing the cutoff lower bound; and the loss of exchangeability in the inhomogeneous setting, which precluded deriving the explicit convergence profile needed to prove sharpness. 

This paper resolves these problems subject to a uniform nondegeneracy assumption by developing a general framework. The key to our approach is a tractable algebraic representation of the system's relative density, constructed from an orthonormal polynomial basis. This representation, along with comonotonic coupling as the non-homogeneous analogue of the monochromatic distribution, enables us to prove the cutoff phenomenon for arbitrary product spaces with inhomogeneous marginals; directly establish the asymptotic sharpness of the bounds without relying on an explicit convergence profile; and generalize the known convergence profile for the homogeneous case from two-spin systems to arbitrary finite state spaces.

The remainder of this paper is organized as follows. We first formally define the generalized model and state our main results in Section 2. We explain our main tools in Section 3, and then establish the sharp estimate on the upper and lower bounds of the mixing time in Sections 4 and 5, respectively. In Section 6, we specialize to the homogeneous case to derive the explicit convergence profile. We also present basic properties of the nonlinear model in the Appendix.  

\section{Notation and Main Results} 

This section formally defines the generalized nonlinear recombination model and presents the main theorems of this paper. We establish the existence of a cutoff phenomenon at time $\log_2 n$ for a broad class of product spaces and provide sharp, quantitative upper and lower bounds.

\subsection{The Model}

Now we explain the model considered in this article, which is a generalization of the model considered in \cite{CaputoLabbeLacoin2025}

\subsubsection*{State space} Let $S = \{s_0, s_1, \dots, s_{k-1}\} \subset \R$ be a set of $k \ge 2$ distinct real-valued spin states. The state space for a system of $n \in \N$ sites is the product space $\Omega_n = S^n$. We denote the set of coordinates by $[n] = \{1, \dots, n\}$.

\subsubsection*{Sequence of marginal distributions:} We define the space of nondegenerate single-site probability distributions on $S$ as:
    \[ 
    \mathcal{P} := \left\{ p: S \to (0,1) : \sum_{l=0}^{k-1} p(s_l) = 1 \right\}. 
    \]    
Then, the marginals for the system are given by an infinite sequence $\mathbf{p} = (p_1, p_2, \dots)$, where $p_i \in \mathcal{P}$ stands for the marginal at site $i$. Therefore, for an $n$-site system, the marginals on $[n]$ are $p^{(n)} = (p_1, \dots, p_n)$.

\subsubsection*{Space of initial distributions} We consider the set of all probability measures on $\Omega_n$ that respect the given marginals:
    \[ 
    \mathcal{P}^{(n)} := \left\{ \mu \in \mathcal{P}(\Omega_n) : \forall i \in [n], \text{ the } i\text{-th marginal of } \mu \text{ is } p_i \right\}. 
    \]

    Any initial distribution $\mu$ is assumed to belong to $\mathcal{P}^{(n)}$. While both $\mu$ and the evolved distribution $\mu_t$ depend on the system size $n$, we suppress this dependence in the notation for convenience.

\subsubsection*{Dynamics} The discrete-time evolution is defined by the initial state $\mu_0 = \mu$ and the recursion $\mu_t = \mu_{t-1} \circ \mu_{t-1}$ for $t \in \N$, where the operator $\circ$ is the averaged uniform recombination (or collision product) defined by:
    \[ 
    \nu_1 \circ \nu_2 = 2^{-n} \sum_{A \subseteq [n]} (\nu_1)_A \otimes (\nu_2)_{A^c},
    \]
    where $(\nu)_A$ denotes the marginal of a measure $\nu$ on the coordinate subset $A$.

  Let $T: \mathcal{P}(\Omega_n) \to \mathcal{P}(\Omega_n)$ be the one-step operator defined by $T(\mu) = \mu \circ \mu$. The system's dynamics can then be expressed as the repeated application of this operator:
\[
    \mu_t = T(\mu_{t-1}) = T^t(\mu_0),
\]
where $T^t$ denotes the $t$-fold composition of $T$.
The family of operators $\{T^t\}_{t \in \mathbb{N}_0}$ forms a discrete-time \textbf{nonlinear semigroup}, satisfying $T^0 = \mathrm{Id}$ and $T^{t+s} = T^t  T^s=T^s  T^t$. This evolution defines a \textbf{nonlinear Markov process}. Unlike a standard (linear) Markov chain where the distribution evolves via a fixed transition operator, here the dynamics explicitly depend on the current distribution $\mu_{t-1}$. Furthermore, in contrast to traditional linear Markov chains, this evolution is not defined at the level of single configurations but is rather an evolution on the space of probability distributions itself.

The operator $\circ$ models a process of sexual recombination, analogous to that in population genetics. The evolution from $\mu_{t-1}$ to $\mu_t$ can be understood through the following intuitive process:
\begin{enumerate}
    \item \textbf{Selection of Parents:} To form a new ``offspring'' configuration, two parent configurations are drawn independently from the population, each according to the distribution $\mu_{t-1}$.
    \item \textbf{Recombination:} For each coordinate $i\in[n]$, the offspring inherits the state (or ``spin'') at coordinate~$i$ from one of the two parents. The choice for each coordinate is made independently and uniformly at random between the two parents, so that all $2^n$ possible coordinate-wise inheritances are equally likely.
    \item \textbf{New Distribution:} The distribution $\mu_t$ is the law of the resulting offspring, averaged over all possible parent selections and recombination choices. Equivalently, for two measures $\nu_1,\nu_2\in\mathcal P(\Omega_n)$,
    \[
        \nu_1\circ\nu_2 \;=\; 2^{-n}\sum_{A\subset[n]} (\nu_1)_A\otimes(\nu_2)_{A^{c}},
    \]
    which is the formal expression of averaging over all $2^n$ ways to partition the coordinates between the two parents.
\end{enumerate}
    
\subsubsection*{Graphical construction} 

    The distribution $\mu_t$ at time $t$ admits a powerful graphical interpretation. It can be viewed as the law of the configuration at the root of a regular binary tree of depth $t$. This graphical representation is fundamental to our analysis. It provides the probabilistic framework for defining the random environment $\xi$ at the leaves of the tree, from which we derive the ``quenched" quantities that are central to all subsequent proofs.

    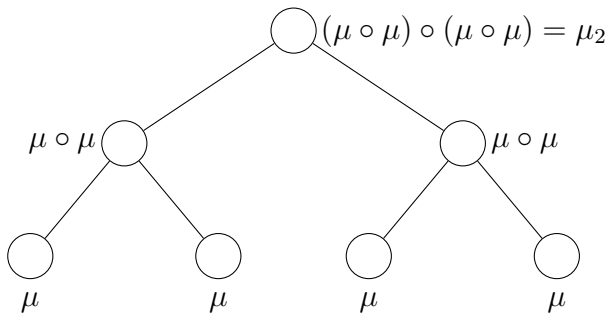
\begin{figure}[htbp]
    \centering
    \begin{tikzpicture}[
        level distance=1.5cm, 
        level 1/.style={sibling distance=4.5cm}, 
        level 2/.style={sibling distance=2.5cm}, 
        every node/.style={circle, draw, minimum size=6mm, inner sep=1pt} 
    ]

    \node[label=right:{$(\mu \circ \mu) \circ (\mu \circ \mu) = \mu_2$}] {}
        child { 
            node[label=left:{$\mu \circ \mu $}] {} 
             child { node[label=below:{$\mu$}] {} }
             child { node[label=below:{$\mu$}] {} }
         }
        child { 
            node[label=right:{$\mu \circ \mu $}] {}
            child { node[label=below:{$\mu$}] {} }
            child { node[label=below:{$\mu$}] {} }
    };

    \end{tikzpicture}
    \caption{The graphical representation of the distribution $\mu_t$ for the case $t=2$. Each internal node represents the collision product of its children's distributions, with the root yielding the final distribution $\mu_2$.} 
    \end{figure}

    Let $N = 2^t$. We consider $N$ independent random configurations $\xi = \{\xi(x) : x=1, \dots, N\}$, where each $\xi(x) \in \Omega_n$ is drawn from the initial distribution $\mu$. These configurations represent the leaves of the binary tree. Let $U_1, \dots, U_n$ be $n$ independent random variables, each uniformly distributed on the set of leaves $\{1, \dots, N\}$, and also independent of $\xi$.

    The configuration at the root, denoted $\sigma^* = (\sigma_1^*, \dots, \sigma_n^*)$, is constructed by assigning to each coordinate $i \in [n]$ the spin from a randomly chosen leaf:
\begin{equation}\label{eqn_xi}
        \sigma_i^* = \xi_i(U_i) \quad \text{for } i \in [n].
\end{equation}    
Note that since the leaf selections $\{U_i\}_{i \in [n]}$ are independent, the components $\sigma_i^*$ are conditionally independent given the environment $\xi$.
    \begin{lemma}
    \label{lem:graphical_construction}
    The law of the random configuration $\sigma^*$ is $\mu_t$. 
    \end{lemma} 

    The proof is essentially identical to the one presented in \cite[Section 2.1]{CaputoLabbeLacoin2025}, and therefore is postponed to Appendix. 

\subsubsection*{Stationary state and convergence}

Denote by $\pi = \bigotimes_{i=1}^n p_i$ the product measure on $[n]$ whose marginal coincides with $p^{(n)}$. Note that, while $\pi=\pi_n$ depends on the system size $n$, we suppress this dependence in the notation for convenience.
 We first show that $\pi$ is the stationary measure for the dynamics defined above. 
    
    \begin{lemma}
    \label{lem:stationary_measure}
      $\pi$ is the unique stationary measure for the dynamics.  In particular, $\mu_t$ converges to $\pi$ as $t\to\infty$ for all $\mu\in \mathcal{P}^{(n)}$.
    \end{lemma}

The proof of this lemma is postponed to the appendix as well. 

The main concern of the current article is the quantification of the convergence described in the previous lemma. 

\subsubsection*{Distance to stationarity}
    To quantify the convergence of the system's distribution over time, we employ a standard metric for probability measures. For any two probability measures $\mu$ and $\nu$ on the state space $\Omega_n$, the total variation distance between them is defined as
    \[
    \|\mu - \nu\|_{\text{TV}} := \sup_{A \subset \Omega_n} |\mu(A) - \nu(A)|,
    \]
    where the supremum is taken over all measurable subsets of $\Omega_n$. For a discrete model as in our case, it is well known (e.g., \cite[Section 4.1]{LevinPeres2017}) that the total-variation distance is a half the $L_1$ distance between the probability mass functions:
    \[
    \|\mu - \nu\|_{\text{TV}} = \frac{1}{2} \sum_{\sigma \in \Omega_n} |\mu(\sigma) - \nu(\sigma)|.
    \]
  
Now, The distance to stationarity at time $t$ is measured by the total variation distance:
    \[ 
    D_n(\mu, t) = \|\mu_t - \pi\|_{\TV}. 
    \]

    The worst-case distance is the supremum over the space of initial distributions:
    \[ 
    D_n(t) = \sup_{\mu \in \mathcal{P}^{(n)}} D_n(\mu, t). 
    \]

\subsection{The Cutoff Phenomenon}
Our main result establishes that the generalized nonlinear recombination system exhibits a cutoff phenomenon at time $t = \log_2 n + O(1)$ for any choice of nondegenerate marginals under the following uniform nondegeneracy condition. 
\medskip 

\noindent \textbf{Assumption 1.} There exists $\delta >0$ such that $ p_i (s)\in [\delta, 1-\delta]$ for all $i\in\mathbb{N}$ and $s\in S$.
\medskip

We remark that this uniform nondegeneracy assumption is immediate if all $p_i$, $i\in\mathbb{N}$ are identical. We, from this moment on, assume Assumption 1 throughout the remainder of the article. 

\begin{theorem}[Cutoff Phenomenon]
\label{thm:cutoff_phenomenon}
For any $\lambda \in \R$, define the time $t_n(\lambda) = \lfloor\log_2 n + \lambda\rfloor$. The worst-case distance satisfies:
\begin{align*}
    \lim_{\lambda \to \infty} \limsup_{n \to \infty} D_n(t_n(\lambda)) &= 0, \\
    \lim_{\lambda \to -\infty} \liminf_{n \to \infty} D_n(t_n(\lambda)) &= 1.
\end{align*}
\end{theorem}

This theorem is a direct consequence of the following quantitatively sharp bounds on the worst-case distance

\begin{theorem}
\label{thm:upper_bound}
Let $(t_n)_{n\in\mathbb{N}}$ be a sequence of integers such that $\lim_{n\to\infty} n 2^{-t_n} = s > 0$. Then, there exists a constant $c=c(k,\delta) > 0$ such that
\begin{equation}
\label{ub}
cs\le \liminf_{n\to\infty} D_n(t_n)\le \limsup_{n \to \infty} D_n(t_n) \le (k-1)s,
\end{equation}
where the left-most inequality holds only for $s \in (0, s_0)$ for some constant $s_0=s_0(k,\delta) > 0$. 
\end{theorem}

Thus, in the asymptotic regime where $n2^{-t_n} \to s$ as $n\to\infty$ for a sufficiently small constant $s>0$, the worst-case distance $D_n(t_n)$ is of order $\Theta(s)$.

\begin{remark} For the bound \eqref{ub}, we shall prove a stronger inequality $D_n(t) \le (k-1)n 2^{-t}$ for all $n$ and $t$.
\end{remark}

The previous theorem is suitable for handling the situation when $s>0$ is small. We next introduce another result which enables us to handle the case when $s>0$ is large. 

\begin{theorem} 
\label{thm: discrete lb}
Let $(t_n)_{n\in\mathbb{N}}$ be a sequence of integers such that $\lim_{n\to\infty} n 2^{-t_n} = s > 0$. Then, there exists a constant $c=c(k,\delta)>0$ such that
\begin{equation}
\label{lb}
1 - 2e^{-cs} \le \liminf_{n \to \infty}D_n(t_n) \le \limsup_{n \to \infty} D_n(t_n) \le 1-\frac{1}{2}e^{-2(k-1)s},
\end{equation}
where the right-most inequality holds only for $s \ge \frac{\log2}{2(k-1)}$. 
\end{theorem}
\begin{remark}
    For the bound (\ref{lb}), we shall prove a stronger inequality $D_n(t) \le 1-\frac{1}{2}e^{-2(k-1)n2^{-t}}$ for $n2^{-t}\ge\frac{\log2}{2(k-1)}$
\end{remark}
The proofs of Theorem~\ref{thm:upper_bound}  and Theorem~\ref{thm: discrete lb} are given in Section~\ref{sec:proof_ub} and Section~\ref{sec:proof_lb}, respectively. At this moment, we assume these theorems and complete the proof of Theorem~\ref{thm:cutoff_phenomenon}.

\begin{proof}[Proof of Theorem~\ref{thm:cutoff_phenomenon}]

From Theorem~\ref{thm:upper_bound}, we have $D_n(t_n(\lambda)) \le n(k-1)2^{-t_n(\lambda)}$. The inequality $\lfloor x \rfloor > x-1$ implies $t_n(\lambda) > \log_2 n + \lambda - 1$. Substituting this into the bound gives:
\[
D_n(t_n(\lambda)) < n(k-1) 2^{-(\log_2 n + \lambda - 1)} = n(k-1) \left(\frac{1}{n} \cdot 2^{-\lambda} \cdot 2\right) = 2(k-1)2^{-\lambda}.
\]
Since this bound is independent of $n$, we have $\limsup_{n \to \infty} D_n(t_n(\lambda)) \le 2(k-1)2^{-\lambda}$.
Taking the limit as $\lambda \to \infty$:
\[
\lim_{\lambda \to \infty} \limsup_{n \to \infty} D_n(t_n(\lambda)) \le \lim_{\lambda \to \infty} 2(k-1)2^{-\lambda} = 0.
\]
As the total variation distance is always non-negative:
\[
\lim_{\lambda \to \infty} \limsup_{n \to \infty} D_n(t_n(\lambda)) = 0.
\]
Using the inequality $\lfloor x \rfloor \le x$, we have $t_n(\lambda) \le \log_2 n + \lambda$. This provides a lower bound for the scaling factor:
\[
n2^{-t_n(\lambda)} \ge n2^{-(\log_2 n + \lambda)} = n \left(\frac{1}{n} \cdot 2^{-\lambda}\right) = 2^{-\lambda}.
\]
By Theorem~\ref{thm: discrete lb}, the distance is bounded below by an increasing function of this factor. Therefore, we obtain:
\[
\liminf_{n \to \infty} D_n(t_n(\lambda)) \ge 1 - 2e^{-c \cdot 2^{-\lambda}}.
\]
Taking the limit as $\lambda \to -\infty$.
\[
\lim_{\lambda \to -\infty} \liminf_{n \to \infty} D_n(t_n(\lambda)) \ge \lim_{\lambda \to -\infty} (1 - 2e^{-c \cdot 2^{-\lambda}}) = 1 - 0 = 1.
\]
Since the total variation distance cannot exceed 1:
\[
\lim_{\lambda \to -\infty} \liminf_{n \to \infty} D_n(t_n(\lambda)) = 1.
\]
Having established both required limits, we completed the proof.
    
\end{proof}

\subsection{The Cutoff Profile for Monochromatic Initial States}

While the preceding theorems establish the existence and location of the cutoff phenomenon for the general inhomogeneous model, they do not describe the shape of the transition within the cutoff window. To provide a finer analysis of this convergence profile, we specialize our focus to the significant case of a homogeneous system, where the single-site marginals are identical ($p_i = p$ for all $i$). For this setting, evolving from a monochromatic initial state, we can derive an explicit characterization of the limiting total variation distance. This result generalizes the profile analysis of \cite{CaputoLabbeLacoin2025} from the two-spin case to arbitrary finite state spaces.

\begin{definition}[Monochromatic Initial Distribution]
The monochromatic initial distribution $\mu$ is a $p$-mixture of pure monochromatic configurations:
\[ 
\mu = \sum_{l=0}^{k-1} p(s_l) \delta_{(s_l, s_l, \dots, s_l)}. 
\]
\end{definition}

For this initial condition, the convergence profile is characterized by the total variation distance between two multivariate normal distributions.

\begin{theorem} 
\label{thm:cutoff_profile}
Let $(t_n)_{n\in\N}$ be a sequence of integers such that  $\lim_{n\to\infty} n 2^{-t_n} = s$ for some constant $s>0$. For the system starting from the monochromatic initial distribution $\mu$, the total variation distance converges to:
\[ 
\lim_{n\to\infty} \|\mu_{t_n} - \pi\|_{\TV} = \|\mathcal{N}(\mathbf{0}, (1+s)I_{k-1}) - \mathcal{N}(\mathbf{0}, I_{k-1})\|_{\TV}, 
\]
where $\mathcal{N}(\mathbf{0}, \Sigma)$ denotes the multivariate normal distribution with mean   $\mathbf{0}$ and covariance matrix $\Sigma$, and $I_{k-1}$ is the $(k-1) \times (k-1)$ identity matrix.
\end{theorem}

Therefore, we can confirm the so-called \textit{profile cutoff}. The proof of this theorem will be given in Section~\ref{sec:explicit profile}.

\begin{remark} (Continuous dynamics) In \cite{CaputoLabbeLacoin2025}, a continuous dynamics for two-spin homogeneous model was also considered. While our presentation focuses on the discrete-time setting, the framework extends straightforwardly to its continuous-time counterpart by integrating the branching process machinery from \cite{CaputoLabbeLacoin2025}. However, we omit a technical derivation of these results in order to focus on the discrete model in a more clear manner. 
\end{remark}

\section{Key Tools}

Before proceeding to the proofs of our main theorems, we introduce the common mathematical framework that underpins all subsequent arguments.

\subsection{Quenched Distribution}

The key observation is that, conditionally on a realization of the environment $\xi$, the components $\sigma_i^*$ are independent by (\ref{eqn_xi}). Therefore, it is natural to divide our analysis into two step: the realization of $\xi$, and the realization of $\sigma_i^*$ conditioning on the realization of $\xi$. Therefore, it is important to understand the distribution of $\sigma_i^*$ conditioned on $\xi$, and the corresponding conditional probability measure, which we call the \textit{quenched distribution}, is denoted by $\mu_t^\xi$. Of course, it is a product measure on $\Omega_n$ with marginals given by the empirical distribution of spins at each site over the leaves:
    \[
    \mu_t^\xi(\sigma) = \prod_{i=1}^n \hat{p}_i^\xi(\sigma_i), \quad \text{where} \quad \hat{p}_i^\xi(s_l) := \frac{1}{N} \sum_{x=1}^N \mathbf{1}_{\{\xi_i(x) = s_l\}}, l=0,1,\dots,k-1
    \]
    The full distribution $\mu_t$ is then recovered by averaging over all possible environments: $\mu_t = \E_\xi[\mu_t^\xi]$.

    \begin{lemma}
    For all $t\ge 0$, we have that $\mu_t = \E_\xi[\mu_t^\xi]$.
    \end{lemma} 

    \begin{proof} This is a direct consequence of Lemma~\ref{lem:graphical_construction}, which establishes that $\mu_t$ is the law of the random configuration $\sigma^*$. The quenched distribution $\mu_t^\xi$ is, by definition, the conditional law of $\sigma^*$ given the environment $\xi$. The identity is therefore an application of the law of total probability. 
    \end{proof}

    The central object of our analysis is the \textit{quenched density}, defined as the Radon-Nikodym derivative of the quenched distribution with respect to the stationary measure:
    \[
    h_t^\xi(\sigma) := \frac{d\mu_t^\xi}{d\pi}(\sigma) = \prod_{i=1}^n \frac{\hat{p}_i^\xi(\sigma_i)}{p_i(\sigma_i)}.
    \]
    The analysis in \cite{CaputoLabbeLacoin2025} hinges on finding a simple, explicit algebraic form for this density. For the homogeneous two-spin (Boolean) case, where spins $\sigma_i \in \{-1, 1\}$ and the stationary marginals are balanced ($p_i(\pm 1) = 1/2$), the density simplifies to a product of linear terms: $h_t^\xi(\sigma) = \prod_{i=1}^n (1 + \sigma_i q^\xi(i))$, where $q^\xi(i)$ is the empirical average of spins at site $i$ over the leaves. They extended this form to the unbalanced two-spin case by re-scaling the spin values from $\{-1, 1\}$ to $\{\sqrt{(1-p_i)/p_i}, -\sqrt{p_i/(1-p_i)}\}$. This specific transformation creates a new spin variable with mean 0 and variance 1 under the stationary measure. In the special case of a two-spin system, this single normalized variable happens to be sufficient to form a complete orthonormal basis for the two-dimensional function space. This is what allows the simple product structure of the density to be preserved.

    This reliance on a simple re-scaling, however, is an artifact of the two-spin system's low dimensionality and fundamentally breaks down for systems with three or more spin states. For $k \ge 3$, a simple, universal algebraic representation is no longer apparent. This obstacle motivates our central methodological contribution: the development of a systematic machinery that provides a tractable algebraic form for the quenched density, regardless of the number of spin states or the specific marginal distributions. We achieve this by expanding the density function onto an orthonormal basis (ONB) of polynomials, constructed for each site's marginal distribution.

\subsection{Orthonormal Polynomial Basis}
\label{sec:basis}

The analysis in this paper utilizes an orthonormal basis for functions defined on the single-site state space $S$. We denote the set of all real-valued functions on $S$ as $V(S)$. As $|S|=k$, $V(S)$ is a $k$-dimensional real vector space, regardless of the site.

For each site $i \in [n]$, we equip this space with a site-specific bilinear form, weighted by the stationary marginal distribution $p_i$:
\[ 
\langle g, h \rangle_{p_i} := \sum_{s \in S} g(s)h(s)p_i(s) = \E_\pi[g(\sigma_i)h(\sigma_i)],
\]
for any functions $g, h \in V(S)$. Assumption 1 ensures this form is a valid inner product, thus defining the inner product space $L^2(S, p_i)$. For brevity, we will often write this as $L^2(p_i)$.

To construct an explicit basis for this space, we begin with the set of monomial functions $\{s^j\}_{j=0}^{k-1} = \{1, s, s^2, \dots, s^{k-1}\}$. These functions form a basis for $V(S)$ by the determinant formula for Vandermonde matrix. 

We now apply the Gram-Schmidt orthonormalization process to this monomial basis with respect to the inner product $\langle \cdot, \cdot \rangle_{p_i}$. This procedure yields, for each site $i$, an orthonormal basis of polynomials $\{f_m^i\}_{m=0}^{k-1}$ for the space $L^2(p_i)$. To make this construction concrete, let us explicitly derive the first two basis functions, $f_0^i$ and $f_1^i$.

\begin{itemize}
\item The first basis function $f_0^i$ is obtained by normalizing the first monomial function, $v_0(s) = s^0 = 1$. Since its squared norm is trivially $1$, we merely have  
\[
f_0^i(s) = \frac{v_0(s)}{\|v_0\|_{L^2(p_i)}} = 1.
\]

\item The second basis function $f_1^i$ is constructed from the second monomial  $v_1(s) = s^1 = s$. First, we make $v_1$ orthogonal to $f_0^i$ by subtracting its projection onto $f_0^i$. Let this orthogonal vector be $w_1$:
\[
w_1(s) = v_1(s) - \langle v_1, f_0^i \rangle_{p_i} f_0^i(s) = s - \E_{p_i}[s]
\]
since $f_0^i(s)=1$. Now, since the squared norm of $w_1$ is the variance of the spin variable, the second orthonormal basis function is the standardized spin variable:
\[
f_1^i(s) = \frac{w_1(s)}{\|w_1\|_{L^2(p_i)}} = \frac{s - \E_{p_i}[s]}{\sqrt{\mathrm{Var}_{p_i}(s)}}.
\]
\end{itemize}

This process continues for higher-order monomials, yielding polynomials $f_m^i$ that are orthogonal to all lower-degree polynomials and involve higher moments of the distribution $p_i$.

The general properties of the basis $\{f_m^i\}_{m=0}^{k-1}$, which are direct consequences of its construction, are as follows:
\begin{itemize}
    \item  \textbf{Zero Mean Property}: For any $m \ge 1$, the function $f_m^i$ is orthogonal to $f_0^i = 1$. This implies that it has zero mean with respect to the stationary measure:
    \begin{equation}
    \label{zero mean}
    \E_\pi[f_m^i(\sigma_i)] = \langle f_m^i, f_0^i \rangle_{p_i} = 0. 
    \end{equation}
    \item \textbf{Uniform Boundedness}: The nondegeneracy assumption, $p_i(s) \ge \delta > 0$, ensures that the basis functions are uniformly bounded. For any $m$ and any $s \in S$, the orthonormality condition implies:
    \begin{equation}
    \label{uniform bdd}
    (f_m^i(s))^2 \delta \le (f_m^i(s))^2 p_i(s) \le \sum_{s' \in S} (f_m^i(s'))^2 p_i(s') = \|f_m^i\|_{L^2(p_i)}^2 = 1.
    \end{equation} 
    Therefore, we have a uniform bound $|f_m^i(s)| \le 1/\sqrt{\delta}$ for all $i, m, s$. \textbf{We emphasize that this is the location where Assumption 1 is crucially required.}
\end{itemize}

\subsection{Quenched Moments and Density Expansion}
\label{sec:quenched_moments}

With the orthonormal basis $\{f_m^i\}$ for each site, we now define the central random variables in our analysis. These variables, which we call the \textit{quenched moments}, are defined as the empirical average of the basis functions evaluated on the spins at the leaves of the binary tree.

\begin{definition}[Quenched Moments]
For each site $i \in [n]$ and basis index $m \in \{0, \dots, k-1\}$, the quenched moment $q_m^\xi(i)$ is defined as
\[
q_m^\xi(i) := \frac{1}{N} \sum_{x=1}^N f_m^i(\xi_i(x)).
\]
\end{definition}

These moments capture the statistical fluctuations of the initial state as seen through the lens of the basis functions. Their fundamental properties are summarized in the following lemma.

\begin{lemma}[Properties of Quenched Moments]
\label{lem:properties of QM}
The quenched moments have the following properties for any site $i \in [n]$:
\begin{enumerate}
    \item For $m=0$: $q_0^\xi(i) = 1$.
    \item For $m \ge 1$, the mean over the environment $\xi$ is zero: $\E_\xi[q_m^\xi(i)] = 0$.
    \item For $m \ge 1$, the second moment is given by: $\E_\xi[(q_m^\xi(i))^2] = 2^{-t}$.
\end{enumerate}
\end{lemma}
\begin{proof}
The proof relies on the properties of the basis $\{f_m^i\}$ and the i.i.d. nature of the leaf configurations $\{\xi(x)\}_{x=1}^N \sim \mu$.
\begin{enumerate}
    \item Since $f_0^i(s)=1$ for all $s$, $q_0^\xi(i) = \frac{1}{N} \sum_{x=1}^N 1 = 1$.
    \item For $m \ge 1$, by linearity of expectation, the fact that $\xi_i(x) \sim p_i$, and the zero mean property (\ref{zero mean}) observed above, we have 
    \[ \E_\xi[q_m^\xi(i)] = \frac{1}{N} \sum_{x=1}^N \E_\xi[f_m^i(\xi_i(x))] = \E_\mu[f_m^i(\sigma_i)] = \E_\pi[f_m^i(\sigma_i)] = 0. \]
   \item For the second moment ($m \ge 1$), expanding the square gives
    \[ \E_\xi[(q_m^\xi(i))^2] = \frac{1}{N^2} \sum_{x,y=1}^N \E_\xi[f_m^i(\xi_i(x)) f_m^i(\xi_i(y))]. \]
    For $x \neq y$, the configurations $\xi(x)$ and $\xi(y)$ are independent, so the expectation of the cross-terms factors into $\E_\xi[f_m^i(\xi_i(x))] \E_\xi[f_m^i(\xi_i(y))] = 0 \cdot 0 = 0$. Only the $N$ diagonal terms ($x=y$) survive:
    \[ \E_\xi[(q_m^\xi(i))^2] = \frac{1}{N^2} \sum_{x=1}^N \E_\xi[(f_m^i(\xi_i(x)))^2] = \frac{N}{N^2} \E_\mu[(f_m^i(\sigma_i))^2] = \frac{1}{N} \|f_m^i\|_{L^2(p_i)}^2 = \frac{1}{N} = 2^{-t}. \]
\end{enumerate}
\end{proof}

The significance of these quenched moments is that they are precisely the coefficients in the expansion of the quenched density $h_t^\xi(\sigma)$.

\begin{proposition}[Density Expansion]
\label{prop:density expansion}
For any $\sigma \in \Omega_n$, the quenched density can be expressed as:
\[
h_t^\xi(\sigma) = \prod_{i=1}^n \left(1 + \sum_{m=1}^{k-1} q_m^\xi(i) f_m^i(\sigma_i)\right).
\]
\end{proposition}
\begin{proof}
We consider the single-site density $h_{t,i}^\xi(s) := \hat{p}_i^\xi(s)/p_i(s)$, which is a function in $L(S)$. Since $\{f_m^i\}_{m=0}^{k-1}$ is an orthonormal basis for $L^2(p_i)$, we can express $h_{t,i}^\xi$ as:
\[ h_{t,i}^\xi(s) = \sum_{m=0}^{k-1} c_m f_m^i(s), \quad \text{where } c_m = \langle h_{t,i}^\xi, f_m^i \rangle_{p_i}. \]
Let us compute the coefficient $c_m$:
\[
c_m = \left\langle \frac{\hat{p}_i^\xi}{p_i}, f_m^i \right\rangle_{p_i} = \sum_{s \in S} \left(\frac{\hat{p}_i^\xi(s)}{p_i(s)}\right) f_m^i(s) p_i(s) = \sum_{s \in S} \hat{p}_i^\xi(s) f_m^i(s).
\]
Substituting the definition of the empirical marginal $\hat{p}_i^\xi(s) = \frac{1}{N} \sum_{x=1}^N \mathbf{1}_{\{\xi_i(x)=s\}}$: 
\[
c_m = \sum_{s \in S} \left(\frac{1}{N} \sum_{x=1}^N \mathbf{1}_{\{\xi_i(x)=s\}}\right) f_m^i(s) = \frac{1}{N} \sum_{x=1}^N \sum_{s \in S} \mathbf{1}_{\{\xi_i(x)=s\}} f_m^i(s) = \frac{1}{N} \sum_{x=1}^N f_m^i(\xi_i(x)).
\]
This final expression is exactly our definition of the quenched moment $q_m^\xi(i)$. Thus, $c_m = q_m^\xi(i)$. The single-site expansion becomes:
\[ h_{t,i}^\xi(\sigma_i) = \sum_{m=0}^{k-1} q_m^\xi(i) f_m^i(\sigma_i) = q_0^\xi(i) f_0^i(\sigma_i) + \sum_{m=1}^{k-1} q_m^\xi(i) f_m^i(\sigma_i) = 1 + \sum_{m=1}^{k-1} q_m^\xi(i) f_m^i(\sigma_i) \]
where the final equality follows from the fact that $f_0^i = 1$ (by construction) and $q_0^\xi(i) = 1$ (by Lemma~\ref{lem:properties of QM} (1)).
The full density is the product over all sites, which completes the proof.
\end{proof}

This expansion provides the fundamental decomposition of the evolved measure and serves as the starting point for proving the main theorems on the cutoff bounds.

\subsection{Comonotonic Coupling and Its Correlation Properties}
\label{subsec:comonotonic}

The proofs of the lower bound and the sharpness of the upper bound rely on constructing specific initial distributions with strong positive correlations. A powerful method for achieving this is the comonotonic coupling, which we formally define below.

\begin{definition}[Comonotonic Coupling]
\label{def:comonotonic} 
Let $\{p_i\}_{i \in I}$ be a collection of single-site probability distributions on $S$, where $I$ is an arbitrary finite index set. A set of random variables $\{\sigma_i\}_{i \in I}$ is said to be \textbf{comonotonically coupled} if there exists a single, common random variable $U \sim \text{Uniform}[0,1]$ such that
\[
\sigma_i = F_i^{-1}(U) \quad \text{for all } i \in I,
\]
where $F_i^{-1}$ is the quantile function (or generalized inverse distribution function) of the marginal distribution $p_i$. This coupling method ensures that all variables $\sigma_i, i\in I$ move together in a highly correlated, non-decreasing manner; if one variable takes a high value (relative to its distribution), all others do as well.
\end{definition}

In their work, Caputo et al. noted that proving a lower bound in the inhomogeneous setting would require one to ``introduce a non-homogeneous analogue of the monochromatic distribution" \cite{CaputoLabbeLacoin2025}. The comonotonic coupling defined above serves precisely this purpose. It is the natural generalization of the monochromatic state—which enforces maximal correlation in the homogeneous case by forcing all spins to be identical—to the inhomogeneous setting. This construction is consistent with the original concept, as it reduces to the monochromatic distribution precisely when all single-site marginals are identical.

The following lemma establishes that the correlation induced by this coupling is not only positive but also uniformly bounded away from zero, a crucial property for our subsequent proofs. For clarity, let us first define $\mathcal{P}_\delta$ as the set of single-site probability distributions satisfying Assumption 1:
$$
\mathcal{P}_\delta := \{ p \in \mathcal{P} \mid p(s) \in [\delta, 1-\delta] \text{ for all } s \in S \}.
$$
For notational convenience, we will denote $f^i_1(\sigma_i)$ simply as $g_i(\sigma_i)$ or $g_i$.

\begin{lemma}
\label{lem:uniform_cov}
There is a uniform bound $\rho= \rho(k,\delta)>0$ such that:
\begin{equation*}
    \mathbb{E}[g_i(\sigma_i)g_l(\sigma_l)] \ge \rho \quad \forall \; p_i, p_l \in \mathcal{P}_\delta
\end{equation*}
where $\sigma_i$ and $\sigma_l$ are coupled comonotonically.
\end{lemma}
\begin{proof} 
The proof is structured in two parts. First, we establish that for any specific pair of marginal distributions, the covariance is strictly positive. Second, we use a compactness argument to show that this positivity implies a uniform lower bound.

\vspace{\baselineskip}
\noindent\textbf{Step 1.} We first prove that for any fixed pair of marginal distributions $(p_i, p_l)$, the correlation $\mathbb{E}[g_i g_l]$ is strictly positive. To that end, without loss of generality, we first order the spin states such that $s_0 < s_1 < \dots < s_{k-1}$. 
The function $g_i$ is strictly increasing since $ g_i(s) = \frac{s - \E_{\sigma_i \sim p_i}[\sigma_i]}{\sqrt{\text{Var}_{\sigma_i\sim p_i}(\sigma_i)}}.$
Let us define the function $h_i: [0,1] \to \R$ as $h_i(u) := g_i(F_i^{-1}(u))$. Since $h_i$ is a composition of two non-decreasing functions, $h_i$ is also non-decreasing. This leads to the pointwise inequality:
\begin{equation*}
    (h_i(x) - h_i(y))(h_l(x) - h_l(y)) \ge 0.
\end{equation*}
Integrating over $[0,1]^2$ gives:
\begin{equation*} 
    \iint_{[0,1]^2} (h_i(x) - h_i(y))(h_l(x) - h_l(y)) \,dx\,dy \ge 0.
\end{equation*}
The nondegeneracy condition $p_i(s) \ge \delta$ implies that if $U \in [0, \delta] $, then $ \sigma_i =\sigma_l= s_0$, and if $U \in [1-\delta, 1] $, then $ \sigma_i=\sigma_l = s_{k-1}$. 

For any point $(x,y) \in [0, \delta] \times [1-\delta, 1]$, we have $h_i(x) - h_i(y) = g_i(s_0) - g_i(s_{k-1}) < 0$ and $h_l(x) - h_l(y) = g_l(s_0) - g_l(s_{k-1}) < 0$.
Since this region has a non-zero measure of $\delta^2 > 0$, and the integrand is non-negative everywhere else, the integral over the entire unit square must be strictly positive:
\[
    \iint_{[0,1]^2} (h_i(x) - h_i(y))(h_l(x) - h_l(y)) \,dx\,dy > 0.
\]
Recalling that $h_i(u) = g_i(F_i^{-1}(u))$ and that the expectation is taken over a common random variable $U \sim \text{Uniform}[0,1]$, we can identify each integral term as an expectation:
\begin{align*}
    \int_0^1 h_i(u)h_l(u) \,du &= \mathbb{E}[h_i(U)h_l(U)] = \mathbb{E}[g_i(\sigma_i)g_l(\sigma_l)] = \mathbb{E}[g_ig_l], \\
    \int_0^1 h_i(u) \,du &= \mathbb{E}[h_i(U)] = \mathbb{E}[g_i(\sigma_i)] = \mathbb{E}[g_i].
\end{align*}
Substituting these back, the expression becomes:
\begin{align*}
    2 \E[g_i g_l] - 2 (\E[g_i])(\E[g_l])  > 0  
\end{align*}
Since $\mathbb{E}[g_i] = \mathbb{E}_\pi[f_1^i(\sigma_i)] = 0$ for any $i$ by (\ref{zero mean}), this simplifies to
$$ \mathbb{E}[g_i g_l] > 0. $$

\vspace{\baselineskip}
\noindent\textbf{Step 2.} Next, we show that the function $\Psi(p_i, p_l) = \mathbb{E}[g_i(\sigma_i)g_l(\sigma_l)]$ is continuous on the compact domain $\mathcal{P}_\delta \times \mathcal{P}_\delta$. The result then follows immediately.

 Let $(p_i^{(n)}, p_l^{(n)})$ be a sequence in $\mathcal{P}_\delta \times \mathcal{P}_\delta$ converging to $(p_i, p_l)$. The covariance for the $n$-th term in the sequence is
\begin{equation*}
    \Psi(p_i^{(n)}, p_l^{(n)}) = \int_0^1 g_{p_i^{(n)}}(F_{p_i^{(n)}}^{-1}(u)) g_{p_l^{(n)}}(F_{p_l^{(n)}}^{-1}(u))\,du.
\end{equation*}
The function $g_i(s)$ depends on the mean $\mathbb{E}_{p_i}[s]$ and variance $\text{Var}_{p_i}(s)$. Both mean and variance are polynomial (hence continuous) functions of the probability vector $p_i$. The quantile function is a finite sum of indicator functions which converges almost everywhere as $(p_i^{(n)}, p_l^{(n)})$ converges to $(p_i, p_l)$. Thus the integrand converges a.e.\ and is bounded by $1/\delta$, as the function $g_i$ is bounded by $1/\sqrt{\delta}$ uniformly over $p_i \in \mathcal{P}_\delta$ by (\ref{uniform bdd}). Applying the Dominated Convergence Theorem gives the desired result.
\end{proof}

\section{Proof of Theorem~\ref{thm:upper_bound}}
\label{sec:proof_ub}

This section is dedicated to the proof of Theorem~\ref{thm:upper_bound}, which establishes the upper bound on the total variation distance and its asymptotic sharpness. 
First, in Subsection~\ref{subsec:ub_proof}, we prove the main upper bound, showing that for all $n$ and $t$, the worst-case distance satisfies
\[
    D_n(t) \le (k-1)n2^{-t}.
\]
Next, in Subsection~\ref{subsec:ub_sharp_proof}, we demonstrate that this linear scaling is asymptotically sharp. To achieve this, we construct a specific initial distribution and show that for any sequence of times $(t_n)_{n\in\mathbb{N}}$ such that $\lim_{n\to\infty} n2^{-t_n} = s > 0$, there exist constants $c=c(k,\delta)>0$ and $s_0=s_0(k,\delta)>0$ such that for all $s \in (0, s_0)$,
\[
    \liminf_{n\to\infty} D_n(t_n) \ge cs.
\]
Together, these results establish both the upper bound and its sharpness as stated in Theorem~\ref{thm:upper_bound}, thereby completing the proof.

\subsection{Proof of the Upper Bound}
\label{subsec:ub_proof}
With the key tools from Section 3 established, our proof adapts the strategy of \cite{CaputoLabbeLacoin2025}, which relies on bounding the distance by the $L^2$-norm of the density fluctuation.
\begin{proof}
Let us define an estimator $\hat{h}_t^\xi$ for the density $h_t^\xi$ as:
$$
\hat{h}_t^\xi = h_t^\xi - \sum_{i=1}^n \sum_{m=1}^{k-1} q_m^{\xi}(i) f_m^i(\sigma_i).
$$

This estimator is constructed by subtracting the linear terms from the density expansion of $h_t^\xi$. A key property of this construction is that its expectation under the stationary measure $\pi$ remains 1, since $\mathbb{E}_\pi[f_m^i(\sigma_i)] = 0$ for $m \ge 1$ due to orthonormality. That is, $\E_\pi(\hat{h}^\xi_t)=\E_\pi(h^\xi_t)=1$. Furthermore, its expectation over the environment $\xi$ is unchanged because the quenched moments $q_m^\xi(i)$ for $m \ge 1$ have zero mean, i.e., $\mathbb{E}_\xi[q_m^\xi(i)]=0$. Thus, $\E_\xi(\hat{h}^\xi_t)=\E_\xi(h^\xi_t)$.

Therefore, we can express the total variation distance using this new estimator:
\begin{equation*} 
D_n(\mu,t)  =\|\mu_t-\pi\|_{TV} = \frac{1}{2} \| \mathbb{E}_\xi[h_t^\xi] - 1 \|_{L^1(\pi)}  
 =\frac{1}{2}  \| \mathbb{E}_\xi [  \hat{h}_t^\xi] - 1  \|_{L^1(\pi)}.
\end{equation*}

By applying Jensen's inequality, we can move the expectation inside the norm, which yields:

\begin{equation*} 
D_n(\mu,t) \le \frac{1}{2} \mathbb{E}_\xi \left[ \left\| \hat{h}_t^\xi - 1 \right\|_{L^1(\pi)} \right].
\end{equation*}

We now proceed by deriving two different bounds for the term $\| \hat{h}_t^\xi - 1 \|_{L^1(\pi)}$.

\medskip

\noindent \textbf{Step 1.} Deriving the first bound of $\| \hat{h}_t^\xi - 1 \|_{L^1(\pi)}$.

First, we use the definition of $\hat{h}_t^\xi$ and the triangle inequality. Since $h_t^\xi$ is a density, its $L^1(\pi)$ norm is 1.
\begin{align*}
\left\| \hat{h}_t^\xi - 1 \right\|_{L^1(\pi)} &= \left\| h_t^\xi - 1 - \sum_{i=1}^n \sum_{m=1}^{k-1} q_m^{\xi}(i) f_m^i(\sigma_i) \right\|_{L^1(\pi)} \\
&\le  \int |h_t^\xi| d\pi + \int \left|\sum_{i=1}^n \sum_{m=1}^{k-1} q_m^{\xi}(i) f_m^i(\sigma_i)\right| d\pi + \int |1| d\pi \\
&= 2 + \int \left|\sum_{i=1}^n \sum_{m=1}^{k-1} q_m^{\xi}(i) f_m^i(\sigma_i)\right| d\pi.
\end{align*}
The integral term can be bounded using the Cauchy-Schwarz inequality. The orthonormality of the basis functions $\{f_m^i\}$ under the product measure $\pi$ greatly simplifies the resulting $L^2$-norm:
\begin{align*}
\int \left|\sum_{i=1}^n \sum_{m=1}^{k-1} q_m^{\xi}(i) f_m^i(\sigma_i)\right| d\pi  \le \sqrt{\int \left(\sum_{i=1}^n \sum_{m=1}^{k-1} q_m^{\xi}(i) f_m^i(\sigma_i)\right)^2 d\pi}   
=  \sqrt{\sum_{i=1}^n \sum_{m=1}^{k-1} (q_m^{\xi}(i))^2}.
\end{align*}
Substituting this back, we obtain our first bound on the distance:
\begin{equation}
\label{bound1}
    D_n(\mu,t)\leq \frac{1}{2} \mathbb{E}_\xi \left[ 2 + \sqrt{\sum_{i=1}^n \sum_{m=1}^{k-1} (q_m^{\xi}(i))^2} \right].
\end{equation}

\medskip

\noindent \textbf{Step 2.} Deriving the second bound of $\| \hat{h}_t^\xi - 1 \|_{L^1(\pi)}$.

Alternatively, we can relate the $L^1$-norm to the $L^2$-norm at an earlier stage. Since $\pi$ is a probability measure, $\|g\|_{L^1(\pi)} \le \|g\|_{L^2(\pi)}$. This gives:
$$
\| \hat{h}_t^\xi - 1 \|_{L^1(\pi)}^2 \le \| \hat{h}_t^\xi - 1 \|_{L^2(\pi)}^2=\E_\pi [ (\hat{h}_t^\xi)^2 ] - 1,
$$

where the final equality holds because $\mathbb{E}_\pi[\hat{h}_t^\xi]=1$. We now compute $\mathbb{E}_\pi [ (\hat{h}_t^\xi)^2 ]$. Expanding the square and using the orthonormality of the basis functions to evaluate the expectations of the cross-terms, we find:

\begin{align*}
\mathbb{E}_\pi [ (\hat{h}_t^\xi)^2 ] &= \mathbb{E}_\pi \left[ \left( h_t^\xi - \sum_{i=1}^n \sum_{m=1}^{k-1} q_m^{\xi}(i) f_m^i(\sigma_i) \right)^2 \right] \\
&= \mathbb{E}_\pi[(h_t^\xi)^2] - 2\mathbb{E}_\pi\left[h_t^\xi \sum_{i,m} q_m^\xi(i)f_m^i(\sigma_i)\right] + \mathbb{E}_\pi\left[\left(\sum_{i,m} q_m^\xi(i)f_m^i(\sigma_i)\right)^2\right] \\
&=\prod_{i=1}^n \left(1 + \sum_{m=1}^{k-1} (q_m^{\xi}(i))^2 \right) - 2 \sum_{i=1}^n \sum_{m=1}^{k-1} (q_m^{\xi}(i))^2 + \sum_{i=1}^n \sum_{m=1}^{k-1} (q_m^{\xi}(i))^2  \\
&= \prod_{i=1}^n \left(1 + \sum_{m=1}^{k-1} (q_m^{\xi}(i))^2 \right) - \sum_{i=1}^n \sum_{m=1}^{k-1} (q_m^{\xi}(i))^2.
\end{align*}

Using the inequality $1+x \le e^x$ to bound the product term, we arrive at:

$$
\mathbb{E}_\pi [ (\hat{h}_t^\xi)^2 ] \le \exp\left(\sum_{i=1}^n \sum_{m=1}^{k-1} (q_m^{\xi}(i))^2 \right) - \sum_{i=1}^n \sum_{m=1}^{k-1} (q_m^{\xi}(i))^2.
$$
This provides our second bound on the distance:
 \begin{equation}
 \label{bound2}
     D_n(\mu,t) \le \frac{1}{2} \E_\xi \left[\sqrt{ \exp\left(\sum_{j=1}^n \sum_{m=1}^{k-1} (q_m^{\xi}(j))^2 \right) - \sum_{i=1}^n \sum_{m=1}^{k-1} (q_m^{\xi}(i))^2 - 1 }\right].
 \end{equation}

\medskip
 
\noindent \textbf{Step 3.} Combining two Bounds

Let $A_\xi = \sum_{j=1}^n \sum_{m=1}^{k-1} (q_m^{\xi}(j))^2$. Our two bounds (\ref{bound1}) and (\ref{bound2}) imply that
$$
D_n(\mu,t) \le \frac{1}{2} \mathbb{E}_\xi \left[ \min\left(2+\sqrt{A_\xi}, \sqrt{e^{A_\xi}-A_\xi-1}\right) \right].
$$

We use the inequality $\min(2+\sqrt{a}, \sqrt{e^a-a-1})\leq 2a$ for any $a>0$. This allows us to uniformly bound the expression inside the expectation:
\begin{align*}
D_n(\mu,t) &\le \frac{1}{2} \mathbb{E}_\xi \left[ 2 \sum_{j=1}^n \sum_{m=1}^{k-1} (q_m^{\xi}(j))^2 \right] = \sum_{j=1}^n \sum_{m=1}^{k-1} \mathbb{E}_\xi [(q_m^{\xi}(j))^2].
\end{align*}

Finally, from the properties of the quenched moments established in Section 3.3, we know that $\mathbb{E}_\xi [(q_m^{\xi}(j))^2] = 2^{-t}$. Substituting this in gives the desired result:

\begin{align*}
D_n(\mu,t) &\le \sum_{j=1}^n \sum_{m=1}^{k-1} 2^{-t} = n (k-1) 2^{-t}.
\end{align*}
This bound holds for any initial distribution $\mu \in \mathcal{P}^{(n)}$. Taking the supremum over all such $\mu$ on the left-hand side gives the bound for the worst-case distance $D_n(t)$ and concludes the proof.

\end{proof}

\subsection{Proof of the Asymptotic Sharpness}
\label{subsec:ub_sharp_proof}
The approach to demonstrating the asymptotic sharpness of the upper bound differs from that of \cite{CaputoLabbeLacoin2025}. Their proof was contingent on an explicit convergence profile for the monochromatic initial state, which is not readily available in the general inhomogeneous setting. We therefore employ a direct method by constructing a globally comonotonic initial distribution and analyzing the leading-order fluctuations of the evolved density. This allows us to establish the required linear lower bound without relying on a specific convergence formula, confirming that the sharpness persists beyond the homogeneous case.

\begin{proof}
To establish a lower bound on the worst-case distance, it suffices to construct a single initial distribution $\mu$ for which the bound holds. We choose the initial distribution $\mu$ constructed via a global comonotonic coupling.
The total variation distance is given by $D_n(\mu,t_n) = \frac{1}{2}\E_\pi[|\rho_{t_n}(\sigma) - 1|]$, where $\rho_{t_n}(\sigma) := \mathbb{E}_\xi[h^\xi_{t_n}(\sigma)]$. We analyze the fluctuation $Q(\sigma) := \rho_{t_n}(\sigma) - 1$, which can be expressed as
\[
Q(\sigma) = \E_\xi \left[ \prod_{i=1}^n \left( 1 + \sum_{m=1}^{k-1} q_m^\xi(i) f_m^i(\sigma_i) \right) - 1 \right].
\]
By expanding the product and using the linearity of expectation, we can decompose $Q(\sigma)$ into a sum of contributions from different degrees of interaction. Let $Q_d(\sigma)$ be the term corresponding to interactions among exactly $d$ sites:
\[
Q_d(\sigma) := \sum_{A \subset [n], |A|=d} \E_\xi \left[ \prod_{i \in A} \left( \sum_{m_i=1}^{k-1} q_{m_i}^\xi(i) f_{m_i}^i(\sigma_i) \right) \right].
\]
Then, $Q(\sigma) = \sum_{d=1}^n Q_d(\sigma)$. Note that $Q_1(\sigma) = 0$ since $\E_\xi[q_m^\xi(i)]=0$ for $m \ge 1$ by Lemma~\ref{lem:properties of QM} (2).
Our proof proceeds via the following three steps:
\begin{enumerate}
    \item Show that $\E_\pi[Q_2^2] = \Omega(s^2)$.
    \item Show that $\E_\pi[Q_2^4] = O(s^4)$.
    \item Show that $\sum_{d=3}^{n}\sqrt{\E_\pi\left[Q_d^2\right]}=O(s^2)$
\end{enumerate}
Once established, the lower bound on the total variation distance follows directly:
\begin{align*}
D_n(\mu,t_n) = \frac{1}{2}\E_\pi[|Q|] &\ge \frac{1}{2}\E_\pi\left[|Q_2|\right] -\frac{1}{2}\sum_{d=3}^{n}\E_\pi\left[|Q_d|\right]\\
&\ge \frac{1}{2} \sqrt{\frac{1}{2}\E_\pi[Q_2^2]} \cdot \p_\pi\left(Q_2^2 > \frac{1}{2}\E_\pi[Q_2^2]\right) -\frac{1}{2}\sum_{d=3}^{n}\sqrt{\E_\pi\left[Q_d^2\right]} \\
&= \Omega(s)-O(s^2)=\Omega(s).
\end{align*}
Let us analyze the scaling of $\E_\pi[Q_d^2]$ for a general degree $d$.
\[
\E_\pi[Q_d^2] = \sum_{A \subset [n], |A|=d} \sum_{\mathbf{m}_A} \left( \E_\xi \left[ \prod_{i \in A} q_{m_i}^\xi(i) \right] \right)^2.
\]
The magnitude of this term is determined by two factors: the number of summands, which is $\binom{n}{d}(k-1)^d = \Theta(n^d)$, and the size of the moment of quenched moments, $(\E_\xi[\dots])^2$.

 To understand the magnitude of the moment $(\E_\xi[\dots])^2$, we substitute the definition of the quenched moments, $q_m^\xi(i) = \frac{1}{N}\sum_{x=1}^N f_m^i(\xi_i(x))$.
\[
\E_\xi\left[\prod_{j=1}^d q_{m_j}^\xi(i_j)\right] = \frac{1}{N^d} \sum_{x_1=1}^N \dots \sum_{x_d=1}^N \E_\xi\left[ \prod_{j=1}^d f_{m_j}^{i_j}(\xi_{i_j}(x_j)) \right].
\]
Consequently, the expectation of the product of functions factors according to the distinct leaf indices chosen from $\{1, \dots, N\}$. Note that if a specific leaf index $y \in \{1, \dots, N\}$ appears only once in the tuple of indices $(x_1, \dots, x_d)$—that is, if it is a singleton—the corresponding term in the sum contains the factor $\E_\xi[f_{m_j}^{i_j}(\xi_{i_j}(y))] = \E_\mu[f_{m_j}^{i_j}(\sigma_{i_j})] = 0$. Such a term contributes nothing to the total sum.

Let $A_d$ denote the total number of non-singleton assignments.
The total number of such assignments is given by the sum:
\begin{align*}
A_d &= \sum_{j=1}^{\lfloor d/2 \rfloor} c(d, j) \cdot P(N, j) 
\end{align*}
where $c(d, j)$ is the number of ways to partition the set of $d$ leaves into $j$ non-empty blocks, with the constraint that each block must contain at least two leaves. The leading coefficient of $A_d$ is $c(d, \lfloor d/2 \rfloor)$. For even $d$, this corresponds to $j=d/2$, where the partitions are necessarily perfect pairings. The number of such pairings is $c(d, d/2) = (d-1)!!$. Thus, for even $d$,
$$
A_d = \frac{d!}{2^{d/2} (d/2)!} N^{d/2} + O(N^{d/2-1}).
$$
\begin{itemize}
    \item For even $d=2l$, recalling the $1/N^d$ normalization from the definition of the quenched moments, the squared expectation scales as:
    \[
    \left(\E_\xi[\dots]\right)^2 = O\left( \left(A_d / N^d\right)^2 \right) = O\left( \left(N^{d/2} / N^d\right)^2 \right) = O(N^{-d}).
    \]
    Combining this with the $\Theta(n^d)$ number of terms, the overall contribution is bounded by $O(s^d).$
    Crucially, the Lemma~\ref{lem:uniform_cov} ensures that the 2-point correlations corresponding to $f_1$ are bounded away from zero. This guarantees that the moment is not just bounded above, but also bounded below in magnitude. Thus, $\E_\pi[Q_d(\sigma)^2]=\Theta(s^d)$.

    \item For odd $d$, The leading coefficient of $A_d$ is $c(d,  (d-1)/2 )$, implying $A_d = O(N^{(d-1)/2})$. Recalling the $1/N^d$ normalization, the squared expectation scales as:
    \[
    \left(\E_\xi[\dots]\right)^2 = O\left( \left(N^{(d-1)/2} / N^d\right)^2 \right) = O(N^{-d-1}).
    \]
    Combining this with the $\Theta(n^d)$ number of terms, the overall contribution is bounded by $O(s^d / N).$
    
\end{itemize}
Note that the first step $\E_\pi[Q_2^2] = \Omega(s^2)$ is proved.

The contribution of the odd-degree terms of $\sum_{d=3}^{n}\sqrt{\E_\pi\left[Q_d^2\right]}$ is asymptotically negligible compared to that of the even-degree terms. The problem of analyzing $\sum_{d=3}^{n}\sqrt{\E_\pi\left[Q_d^2\right]}$ thus reduces to analyzing the sum of the even-degree components: 
\[
\sum_{l=2}^{n}\sqrt{\E_\pi\left[Q_{2l}^2\right]} = \sum_{l=2}^{\infty}\sqrt{\E_\pi\left[Q_{2l}^2\right]},
\]
where we formally extended the sum to infinity by defining $Q_d(\sigma)$ is identically zero for $d > n$. 
To apply the Reverse Fatou's Lemma, we must find a summable sequence $\{b_l\}_{l \ge 2}$ and an integer $n_0$ such that for all $n \ge n_0$ and all $l \ge 2$, the inequality $\sqrt{\mathbb{E}_\pi[Q_{2l}(\sigma)^2]} \le b_l$ holds.
Observe that
\[
\E_\pi[Q_{2l}(\sigma)^2] \le \binom{n}{2l}(k-1)^{2l} \left( \frac{1}{N^{2l}} \sum_{j=1}^{l} c(2l, j) \cdot P(N, j) \cdot \frac{1}{\delta^{l}}\right)^2.
\]
By the definition of a limit, for any $\epsilon > 0$, there exists an integer $n_0=n_0(\epsilon)$ such that for all $n \ge n_0$, we have the uniform bound $n/N < s+\epsilon$. Using this, along with the relations $\binom{n}{2l} \le \frac{n^{2l}}{(2l)!}$ and $P(N,j) \le N^l$, we can establish a dominating sequence that is independent of $n$. For all $n \ge n_0$, we have:
\[
\E_\pi[Q_{2l}(\sigma)^2] \le \frac{(s+\epsilon)^{2l}}{(2l)!} \left(\frac{k-1}{\delta}\right)^{2l} \left( \sum_{j=1}^{l} c(2l, j) \right)^2 =: b_l^2(s)
\]
To confirm that $\sum b_l(s)$ converges for a certain range of $s$, we apply the ratio test:
\begin{align*}
\lim_{l \to \infty} \frac{b_{l+1}(s)}{b_l(s)} 
&= (s+\epsilon) \left(\frac{k-1}{\delta}\right)  \lim_{l \to \infty}\frac{(2l+1)!!}{(2l-1)!!} \sqrt{\frac{(2l)!}{(2l+2)!}} \\
&= (s+\epsilon) \left(\frac{k-1}{\delta}\right)  \lim_{l \to \infty} \sqrt{\frac{2l+1}{2l+2}} = (s+\epsilon) \left(\frac{k-1}{\delta}\right).
\end{align*}
Hence the radius of convergence is $s< \delta/(k-1)-\epsilon$, so the series $\sum b_l(s)$ converges absolutely for $s$ in this range. If we assume $s < 2\delta/(3k-3)$, and choose $\epsilon = s/2$, then the condition for convergence from the ratio test, $(s+\epsilon)(k-1)/\delta < 1$, is satisfied. This choice fixes our dominating sequence $\{b_l\}$ for all $n \ge n_0(\epsilon)$ and guarantees its summability. Thus, by the Reverse Fatou's Lemma, we prove the third step:
\[
\limsup_{n\to\infty} \sum_{l=2}^{\infty} \sqrt{\mathbb{E}_\pi[Q_{2l}(\sigma)^2]} \le \sum_{l=2}^{\infty} \limsup_{n\to\infty} \sqrt{\mathbb{E}_\pi[Q_{2l}(\sigma)^2]} \le \sum_{l=2}^{\infty} b_l = O((3s/2)^2)=O(s^2).
\]
The bound $O(s^2)$ follows.

 It remains to show the second condition $\E_\pi[Q_2^4] = O(s^4)$.
\begin{align*}
    Q_2(\sigma) 
    &= \sum_{1 \le i < j \le n} \sum_{m_1=1}^{k-1} \sum_{m_2=1}^{k-1} \E_\xi[q_{m_1}^\xi(i)q_{m_2}^\xi(j)] f_{m_1}^i(\sigma_i)f_{m_2}^j(\sigma_j) \\
    &\le \frac{1}{N\delta}\sum_{1 \le i < j \le n} \sum_{m_1=1}^{k-1} \sum_{m_2=1}^{k-1}f_{m_1}^i(\sigma_i)f_{m_2}^j(\sigma_j) := M(\sigma).
\end{align*}
Since $\E_\pi$ is a product measure, the expectation of any term in the expansion of $M(\sigma)^4$ factors by site. Given that $\E_{\pi}[f_m^i(\sigma_i)]=0$ for $m \ge 1$, any term where a site index appears with multiplicity one vanishes. Thus the dominant contribution therefore comes from terms involving four distinct site indices with multiplicity two, which can be chosen in $O(n^4)$ ways. Thus $\E_\pi[Q_2^4] = O(s^4)$.
\end{proof}

\section{Proof of Theorem~\ref{thm: discrete lb}}
\label{sec:proof_lb}

This section provides the proof for Theorem~\ref{thm: discrete lb}, which establishes the exponential lower bound and demonstrates its sharpness.
In Subsection~\ref{subsec:lb_proof}, we construct the main lower bound. Specifically, for any sequence of integers $(t_n)_{n\in\mathbb{N}}$ such that $\lim_{n\to\infty} n 2^{-t_n} = s > 0$, we prove that
\[
    \liminf_{n \to \infty}D_n(t_n) \ge 1 - 2e^{-cs}
\]
for some constant $c = c(k, \delta) > 0$.
Subsequently, in Subsection~\ref{subsec:lb_sharp_proof}, we prove the complementary upper bound that establishes the sharpness of the functional form of our lower bound. We show that for $n2^{-t}\ge\frac{\log2}{2(k-1)}$, the distance is bounded by
\[
    D_n(t) \le 1-\frac{1}{2}e^{-2(k-1)n2^{-t}}.
\]
These two bounds, taken together, immediately yield the statement of Theorem~\ref{thm: discrete lb}.
\subsection{Proof of the Lower Bound}
\label{subsec:lb_proof}
This subsection proves the cutoff lower bound by adapting the ``test event" strategy from \cite{CaputoLabbeLacoin2025}. The core idea is to construct a specific initial distribution $\mu$ and an event $A$ such that the evolved measure $\mu_{t_n}(A)$ is close to 1, while the stationary measure $\pi(A)$ is close to 0. In their two-spin setting, Caputo et al. achieved this by partitioning the coordinates into ``baskets" and using a block-wise monochromatic initial distribution. This setup ensures that the high-magnetization indicators for each basket behave as i.i.d. Bernoulli random variables, making the problem amenable to a standard Chernoff bound analysis. We follow this template, using the comonotonic coupling as the required inhomogeneous analogue of the monochromatic state. The proof is organized as follows: we first define the partitioning and the event, then construct the initial distribution, and finally analyze the probability of the event under both the stationary and evolved measures.

We note that establishing the cutoff lower bound requires analyzing the regime where $s = \lim_{n \to \infty } n2^{-t_n}$ is large. Therefore, our proof will focus on the case where $s$ is sufficiently large. More formally, we will first prove the result under the assumption that $s > C_0$ for an arbitrary constant $C_0 = C_0(k, \delta) > 0$ that is independent of $n$. For the remaining range $s \in (0, C_0]$, the claimed lower bound $1 - 2e^{-cs}$ can be made negative by choosing a sufficiently small constant $c > 0$. Since the total variation distance is always non-negative, the inequality holds trivially in this case. Thus, we can proceed without loss of generality by assuming $s > C_0$.

\subsubsection{Partitioning and Event Definition}
We partition the set of coordinates $[n]$ into $a = \lfloor n/b \rfloor$ blocks (baskets) of size $b=C_02^{t_n}$, denoted $B_1, \dots, B_a$, note that for sufficiently large $n$, we have $a\ge1$ since $s> C_0$. The remaining $n - ab$ coordinates form a leftover block $B_{\text{rem}}$. For each block $j \in \{1, \dots, a\}$, we define its squared magnetization as
\begin{equation*}
    \Xi_j := \left( \sum_{i \in B_j} f^i_1(\sigma_i) \right)^2.
\end{equation*}
For notational convenience, we will denote $f^i_1(\sigma_i)$ simply as $g_i(\sigma_i)$ or $g_i$. 
Let $X_j = \ind_{\{\Xi_j \ge 20b\}}$ be the indicator variable for a block having high magnetization. Let $Z = \sum_{j=1}^a X_j$. We define the event $A$ as
\begin{equation*}
    A := \left\{ Z \ge \frac{a}{15} \right\}.
\end{equation*}
This is the event that at least a fraction $1/15$ of the blocks exhibit high magnetization.

\subsubsection{Initial Distribution $\mu$}
We construct the initial distribution $\mu$ to be a product measure over the blocks, where spins within each block are strongly correlated.
\begin{equation*}
    \mu = \left( \bigotimes_{j=1}^a \mu_j \right) \otimes \pi_{\text{rem}},
\end{equation*}
where $\pi_{\text{rem}}$ is the stationary product measure on the leftover block $B_{\text{rem}}$. Each measure $\mu_j$ on the block $B_j$ is the comonotonic coupling of the marginals $(p_i)_{i \in B_j}$.
The correlation induced by this coupling is strictly positive and uniformly bounded away from zero, as established in Lemma~\ref{lem:uniform_cov}. Note that by the Cauchy-Schwarz inequality and the orthonormality of the basis functions, we have the bounds 
\[
0<\rho \le\mathbb{E}[g_ig_l] \le (\mathbb{E}[g_i^2])^{1/2}(\mathbb{E}[g_l^2])^{1/2}=(\langle g_i, g_i \rangle_{p_i})^{1/2}(\langle g_l, g_l \rangle_{p_l})^{1/2} = 1
\] 
which ensures $0 < \rho \le 1$. This constant is the basis for defining the threshold $C_0 = 80/\rho$. For notational convenience, we will denote $f^i_1(\sigma_i)$ simply as $g_i(\sigma_i)$ or $g_i$.

\subsubsection{Analysis under the Stationary Distribution $\pi$}
Under the stationary measure $\pi$, the random variables $g_i(\sigma_i)$ are independent with mean 0 and variance 1.
By Chebyshev's inequality,
\begin{equation*}
    \pi(\Xi_j \ge 20b) = \pi\left( \left(\sum_{i \in B_j} g_i(\sigma_i)\right)^2 \ge 20b \right) \le \frac{\text{Var}_\pi(\sum g_i)}{20b} = \frac{b}{20b} = \frac{1}{20}.
\end{equation*}
The indicator variables $X_j$ are i.i.d. Bernoulli random variables with parameter $p = \pi(X_j=1) < 1/20$. By Chernoff's bound, for sufficiently large $n$, there exists a constant $c_1 > 0$ such that $\pi(A) \le e^{-c_1 a}$. Taking the limit $n\to\infty$ with $n2^{-{t_n}} \to s$, by adjusting the constant $c_1$, we obtain the required inequality:
\begin{equation}
\label{lem:lb stationary}
    \limsup_{n\to \infty}{\pi(A)} \le e^{-c_1s} 
\end{equation}

\subsubsection{Analysis under the Evolved Distribution $\mu_{t_n}$}
The first moment of the squared magnetization under $\mu_{t_n}$ is:
\begin{align*}
    \E_{\mu_{t_n}}[\Xi_j] &= \E_{\mu_{t_n}}\left[ \left( \sum_{i \in B_j} g_i(\sigma_i) \right)^2 \right]
    = \sum_{i \in B_j} \E_{\mu_{t_n}}[g_i^2] + \sum_{i \ne l \in B_j} \E_{\mu_{t_n}}[g_i g_l] \\
    &= b + 2^{-{t_n}} \sum_{i \ne l \in B_j}\E_\mu(g_ig_l) 
    \ge b + b(b-1)2^{-{t_n}} \rho.
\end{align*}
By choosing $C_0=80/\rho$, we have $b=C_02^{t_n}\ge80\cdot2^{t_n}/\rho\ge80$. Thus we can ensure $\E_{\mu_{t_n}}[\Xi_j] \ge 40b$.
The second moment $\E_{\mu_{t_n}}(\Xi_j^2)$ can be expanded by categorizing the terms in $\E_{\mu_{t_n}}\left( \left( \sum_{i \in B_j} g_i(\sigma_i) \right)^4 \right)$ based on how the indices coincide:

\begin{align*}
\E_{\mu_{t_n}}[\Xi_j^2] = & \sum_{\substack{i,k,l,m \in B_j \\ \text{distinct}}} \E_{\mu_{t_n}}[g_i g_k g_l g_m]  + 6 \sum_{\substack{i,k,l \in B_j \\ \text{distinct}}} \E_{\mu_{t_n}}[g_i^2 g_k g_l]  + 3 \sum_{\substack{i,k \in B_j \\ i \neq k}} \E_{\mu_{t_n}}[g_i^2 g_k^2]  \\ &+ 4 \sum_{\substack{i,k \in B_j \\ i \neq k}} \E_{\mu_{t_n}}[g_i^3 g_k]  + \sum_{i \in B_j} \E_{\mu_{t_n}}[g_i^4].
\end{align*}
By considering the underlying fragmentation process, the expectation $\E_{\mu_{t_n}}[\Xi_j^2]$ can be expanded in terms of the initial distribution $\mu$ as follows:

\begin{equation*}
\begin{split}
    \E_{\mu_{t_n}}[\Xi_j^2] = & \sum_{\substack{i,k,l,m \in B_j \\ \text{distinct}}} \bigg[ 2^{-3{t_n}} \E_\mu[g_i g_k g_l g_m]  + 3(1 - 2^{-{t_n}}) 2^{-2{t_n}} \E_\mu[g_i g_k] \E_\mu[g_l g_m]  \bigg] \\
    & + 6 \sum_{\substack{i,k,l \in B_j \\ \text{distinct}}} \Big[ 2^{-2{t_n}} E_\mu[g_i^2 g_k g_l] + (2^{-{t_n}} - 2^{-2{t_n}}) \E_\mu[g_i^2] \E_\mu[g_k g_l] \Big] \\
    & + 3 \sum_{\substack{i,k \in B_j \\ i \neq k}} \Big[ 2^{-{t_n}} \E_\mu[g_i^2 g_k^2] + (1 - 2^{-{t_n}}) \E_\mu[g_i^2] \E_\mu[g_k^2] \Big] \\
    & + 4 \sum_{\substack{i,k \in B_j \\ i \neq k}} \Big[ 2^{-{t_n}} \E_\mu[g_i^3 g_k] \Big]  + \sum_{i \in B_j} \E_\mu[g_i^4].
\end{split}
\end{equation*}
To bound this expression, we use the uniform boundedness of the basis functions, $|g_i| \le 1/\sqrt{\delta}$. This yields:

\begin{equation*}
\begin{split}
\E_{\mu_{t_n}}[\Xi_j^2] \le & \ b(b-1)(b-2)(b-3) \cdot 2^{-3{t_n}} \frac{1}{\delta^2} + (1 - 2^{-{t_n}}) 2^{-2{t_n}} \cdot 3 \left(\sum_{\substack{i,l \in B_j \\ i \neq l}}\E_\mu[g_i g_l] \right)^2 \\
& + 6 b(b-1)(b-2) \left[ 2^{-2{t_n}} \frac{1}{\delta^2} \right] + 6b(2^{-{t_n}} - 2^{-2{t_n}}) \sum_{\substack{i,l \in B_j \\ i \neq l}}\E_\mu[g_i g_l] \\
& + 3b(b-1) \left[ 2^{-{t_n}} \frac{1}{\delta^2} + (1 - 2^{-{t_n}}) \right] + 4 b(b-1) \left[ 2^{-{t_n}} \frac{1}{\delta^2} \right] + \frac{b}{\delta^2}
\end{split}
\end{equation*}
A comparison of the leading-order terms in the expansions of $E_{\mu_{t_n}}[\Xi_j^2]$ and $(E_{\mu_{t_n}}[\Xi_j])^2$ directly yields the inequality $E_{\mu_{t_n}}[\Xi_j^2] \le 3 (E_{\mu_{t_n}}[\Xi_j])^2 + O(b)$. For sufficiently large $n$, the $O(b)$ term is absorbed by the $0.1 (E_{\mu_{t_n}}[\Xi_j])^2$ term. We thus arrive at the following bound:
\begin{equation*}
\E_{\mu_{t_n}}[\Xi_j^2] \le 3.1 (\E_{\mu_{t_n}}[\Xi_j])^2.
\end{equation*}
Applying the Paley-Zygmund inequality then provides a bound for $\E_{\mu_{t_n}}(X_j=1)$:
\begin{align*}
    \mu_{t_n}(\Xi_j \ge 20b) \ge \mu_{t_n}(\Xi_j \ge \frac{1}{2} \E_{\mu_{t_n}}[\Xi_j]) 
    \ge \frac{(\E_{\mu_{t_n}}[\Xi_j])^2}{4\E_{\mu_{t_n}}[\Xi_j^2]} 
    \ge \frac{1}{12.4}.
\end{align*}
By Chernoff's bound, for sufficiently large $n$, we find that for some constant $c_2 > 0$, $\mu_{t_n}(A^c) \le e^{-c_2 a}$. Taking the limit $n\to\infty$ with $n2^{-{t_n}} \to s$, we obtain the required lower bound for $\mu_{t_n}(A)$:
\begin{equation}
\label{lem:lb evolved}
    \liminf_{n\to \infty}\mu_{t_n}(A) \ge 1-e^{-c_2 s} 
\end{equation}

\begin{proof}[Proof of the cutoff lower bound]
We now combine the established bounds to prove (\ref{lb}). By the definition of total variation distance, for the specific event $A$ constructed in Section 5.1, we have:
\[
D_n({t_n})\ge D_n(\mu,{t_n}) = \sup_{B \subseteq \Omega_n} \lvert \mu_{t_n}(B) - \pi(B) \rvert \ge \lvert \mu_{t_n}(A) - \pi(A) \rvert.
\]
Taking the $\liminf$ as $n\to\infty$ on both sides, we can apply the bounds derived in (\ref{lem:lb stationary}) and (\ref{lem:lb evolved}):
\begin{align*}
    \liminf_{n\to\infty} D_n({t_n}) &\ge \liminf_{n\to\infty} (\mu_{t_n}(A) - \pi(A)) \ge \liminf_{n\to\infty} \mu_{t_n}(A) - \limsup_{n\to\infty} \pi(A) \\
    &\ge (1 - e^{-c_2 s}) - (e^{-c_1 s}) = 1 - (e^{-c_1 s} + e^{-c_2 s}).
\end{align*}
Let $c = \min(c_1, c_2)$. Then $e^{-c_1 s} + e^{-c_2 s} \le 2e^{-cs}$. This yields a bound of the form $1 - 2e^{-cs}$ for $s > C_0$.
As argued at the outset of this proof, the bound holds trivially for $s \in (0, C_0]$ by adjusting $c$ if necessary. Therefore, the inequality holds for all $s > 0$. This completes the proof.
 \end{proof}
 
\subsection{Proof of the Asymptotic Sharpness}
\label{subsec:lb_sharp_proof}

\begin{proof}
 Instead of using only the Schwarz inequality to bound the $L^1(\pi)$ norm by the $L^2(\pi)$ norm, we employ the finer inequality presented in \cite[Appendix B]{CaputoLabbeLacoin2025}. This inequality states that for any density $f$ with respect to $\pi$,
\begin{equation} \label{eq:phi_inequality}
    \frac{1}{2} \|f - 1\|_{L^1(\pi)} \le \phi(\|f - 1\|_{L^2(\pi)}),
\end{equation}
where the function $\phi: \mathbb{R}_+ \to [0,1)$ is defined as
\[
\phi(x) = 
\begin{cases}
    x/2, & \text{if } x < 1, \\
    1 - \frac{1}{1+x^2}, & \text{if } x \ge 1.
\end{cases}
\]
The total variation distance can be bounded using the inequality \eqref{eq:phi_inequality}:
\begin{align*}
    \|\mu_t - \pi\|_{\mathrm{TV}} &= \frac{1}{2} \|\E_\xi[h_t^\xi - 1]\|_{L^1(\pi)} \le \frac{1}{2} \E_\xi[\|h_t^\xi - 1\|_{L^1(\pi)}] \le \E_\xi[\phi(\|h_t^\xi - 1\|_{L^2(\pi)})].
\end{align*}
The squared $L^2$ norm of the quenched density fluctuation is given by
\[
\|h_t^\xi - 1\|_{L^2(\pi)}^2 = \E_\pi[(h_t^\xi)^2] - 1 = \prod_{i=1}^n \left(1 + \sum_{m=1}^{k-1} (q_m^\xi(i))^2\right) - 1.
\]
Let us define the random variable $A_\xi := \sum_{i=1}^n \sum_{m=1}^{k-1} (q_m^\xi(i))^2$. Using the inequality $1+x \le e^x$, we can bound the norm:
\[
\|h_t^\xi - 1\|_{L^2(\pi)}^2 \le \exp(A_\xi) - 1.
\]
Since $\phi$ is a non-decreasing function, we have
\[
\|\mu_t - \pi\|_{\mathrm{TV}} \le \E_\xi[\phi(\sqrt{\exp(A_\xi) - 1})].
\]
The expectation of $A_\xi$ is $S = n(k-1)2^{-t}$. By Markov's inequality, we have $\Prob(A_\xi \ge 2S) \le 1/2$.
We now split the expectation based on the event $\{A_\xi \ge 2S\}$. Since $\phi$ is non-decreasing and bounded by 1, and provided $n2^{-t}\ge \frac{\log2}{2(k-1)}$, we obtain
\begin{align*}
    \E_\xi[\phi(\sqrt{\exp(A_\xi) - 1})] &= \E_\xi[\phi(\dots)\mathbf{1}_{\{A_\xi \ge 2S\}}] + \E_\xi[\phi(\dots)\mathbf{1}_{\{A_\xi < 2S\}}] \\
    &\le \Prob(A_\xi \ge 2S) \cdot 1 + \E_\xi[\phi(\sqrt{\exp(A_\xi) - 1})\mathbf{1}_{\{A_\xi < 2S\}}] \\
    &\le \Prob(A_\xi \ge 2S) + \Prob(A_\xi < 2S) \cdot \phi(\sqrt{\exp(2S) - 1}).\\
    &\le 1-\frac{1}{2}\exp(-2(k-1)n2^{-t}).
\end{align*}
\end{proof}

\section{The Cutoff Profile for Monochromatic Initial States}
\label{sec:explicit profile}
In this section, we focus on a \textbf{homogeneous system}, where all single-site marginal distributions are identical, i.e., $p_i = p$ for all sites $i \in [n]$. We begin by formally defining the key vector notations.

With the machinery developed in Section 3, the proof of the cutoff profile follows a path similar to that in \cite{CaputoLabbeLacoin2025}, which relies on analyzing the asymptotic distribution of the evolved density.

\begin{definition}[Quenched Moment Vector]
For a given realization of the environment $\xi = \{\xi(x)\}_{x=1}^N$, the quenched moment for the $m$-th basis function is constant across all sites $i$, since $\xi_i(x)$ is independent of $i$. We define the $(k-1)$-dimensional quenched moment vector $\qvec^\xi$ as:
\[
\qvec^\xi := \frac{1}{N} \sum_{x=1}^N \fvec(\xi(x)),
\]
where $\fvec(s) := (f_1(s), \dots, f_{k-1}(s))^T$. Thus, $\qvec^\xi$ is a random vector representing an average of $N$ i.i.d. random vectors.
\end{definition}

\begin{definition}[Evolved Density]
The density $\rho_t(\sigma)$ is given by:
\[
\rho_t(\sigma) = \E_\xi \left[ \prod_{i=1}^n \left(1 + \qvec^\xi \cdot \fvec(\sigma_i)\right) \right].
\]
\end{definition}
 Our analysis relies on two Central Limit Theorems. First, for the stationary fluctuations, the vector $\bar{\mathbf{f}}_n(\sigma) := \frac{1}{\sqrt{n}}\sum \mathbf{f}(\sigma_i)$ converges in distribution to a standard normal $\mathcal{N}(0, I_{k-1})$. Second, for the quenched moments, the scaled vector $\sqrt{n}\mathbf{q}^\xi$ converges in distribution to $\mathcal{N}(0, sI_{k-1})$, as it is an average of $N=2^{t_n}$ i.i.d. vectors with covariance $I_{k-1}$ and the scaling factor is $\sqrt{n/N} \to \sqrt{s}$.
\begin{remark}
    
        The analytical framework of this section hinges on a crucial simplification: reducing the analysis from the high-dimensional distribution of the configuration $\sigma \in \Omega_n$ to the distribution of a single summary statistic, the stationary fluctuation vector $\bar{\mathbf{f}}_n(\sigma)$. This reduction is possible because the evolved measure $\mu_t$ is \textbf{exchangeable}---the joint probability of a configuration is invariant under any permutation of its site indices.
        
        More precisely, this means that conditional on a given vector of a spin proportion, the measure $\mu_t$ is uniform over the set of all configurations sharing that proportion. Consequently, the quenched density $h_t^\xi(\sigma)$ depends on a configuration $\sigma$ only through its empirical spin counts $\{n_l(\sigma)\}$. It is this fundamental symmetry that allows us to focus our analysis on the fluctuations of these empirical counts, which are fully captured by the vector $\bar{\mathbf{f}}_n(\sigma)$.
        
        This approach is not viable for the general inhomogeneous case. There, the lack of exchangeability stems from the fact that the site-specific marginals $p_i$ give rise to site-dependent quenched moments $q_m^\xi(i)$. As a result, the quenched density $h_t^\xi(\sigma)$ depends on the specific spin at each individual site, not just the overall counts, preventing such a dimensional reduction and forcing a direct analysis of the entire configuration's joint distribution. \hfill$\lozenge$
    
\end{remark}

To analyze the asymptotic behavior of the evolved density $\rho_t(\sigma)$ with its single summary statistic $\bar{\mathbf{f}}_n(\sigma)$, we separate the quenched density $h_t^\xi(\sigma)$ into its dependencies on the quenched moment vector $\qvec^\xi$ and the stationary fluctuation vector $\bar{\fvec}_n(\sigma)$.
The logarithm of the quenched density is given by:
\begin{equation*}
\log h_t^\xi(\sigma) = \sum_{i=1}^n \log\left(1 + \qvec^\xi \cdot \fvec(\sigma_i)\right).
\end{equation*}
Let $n_l(\sigma)$ be the number of sites $i$ where the spin is $s_l$, so that $\sum_{l=0}^{k-1} n_l(\sigma) = n$. We can rewrite the log-density in terms of these counts:
\begin{equation*}
\log h_t^\xi(\sigma) = \sum_{l=0}^{k-1} n_l(\sigma) \log\left(1 + \qvec^\xi \cdot \fvec(s_l)\right).
\end{equation*}
Let $\Gvec(s) = (f_0(s), \dots, f_{k-1}(s))^T$. The corresponding stationary fluctuation vector is $\bar{\Gvec}_n(\sigma) = \frac{1}{\sqrt{n}}\sum_{i=1}^n \Gvec(\sigma_i)$. Note that its first component is $(\bar{\Gvec}_n(\sigma))_0 = \sqrt{n}$, and the remaining components form the vector $\bar{\fvec}_n(\sigma)$.

 The fluctuation vector can be expressed using the empirical frequencies $\tilde{p}_l(\sigma) = n_l(\sigma)/n$:
\[
\bar{\Gvec}_n(\sigma) = \frac{1}{\sqrt{n}} \sum_{l=0}^{k-1} n_l(\sigma) \Gvec(s_l) = \sqrt{n} \sum_{l=0}^{k-1} \tilde{p}_l(\sigma) \Gvec(s_l).
\]
Let $\mathcal{M}$ be the $k \times k$ matrix whose columns are the vectors $\Gvec(s_l)$, i.e., $\mathcal{M}_{ml} = f_m(s_l)$. Since $\{f_m\}_{m=0}^{k-1}$ forms a basis for the space of functions on $S$, the matrix $\mathcal{M}$ is invertible. Note that $\mathcal{M}^{-1} = P \mathcal{M}^T $ where $P$ is the diagonal matrix with entries $P_{ll} = p(s_l)$. This allows us to express the empirical frequencies in terms of the fluctuation vector:
\[
\sqrt{n} \, \tilde{\mathbf{p}}(\sigma) = \mathcal{M}^{-1} \bar{\Gvec}_n(\sigma),
\]
where $\tilde{\mathbf{p}}(\sigma) = (\tilde{p}_0(\sigma), \dots, \tilde{p}_{k-1}(\sigma))^T$.

 Substituting this back into the expression for the log-density, we obtain the desired separation:
\begin{align*}
\log h_t^\xi(\sigma) &= n \sum_{l=0}^{k-1} \tilde{p}_l(\sigma) \log\left(1 + \qvec^\xi \cdot \fvec(s_l)\right) = \sqrt{n} \sum_{l=0}^{k-1} \left(\mathcal{M}^{-1} \bar{\Gvec}_n(\sigma)\right)_l \log\left(1 + \qvec^\xi \cdot \fvec(s_l)\right) \\
&= \sum_{m=0}^{k-1} (\bar{\Gvec}_n(\sigma))_m \left( \sqrt{n} \sum_{l=0}^{k-1} (\mathcal{M}^{-1})_{ml} \log\left(1 + \qvec^\xi \cdot \fvec(s_l)\right) \right) = \beta_n^\xi + \boldsymbol{\alpha}_n^\xi \cdot \bar{\fvec}_n(\sigma),
\end{align*}
where we have defined the scalar $\beta_n^\xi$ and the $(k-1)$-dimensional vector $\boldsymbol{\alpha}_n^\xi$ as:
\begin{align*}
\beta_n^\xi &:= n \sum_{l=0}^{k-1} p(s_l) \log\left(1 + \qvec^\xi \cdot \fvec(s_l)\right)\\
\boldsymbol{\alpha}_n^\xi &:= \sqrt{n} \sum_{l=0}^{k-1} p(s_l) \fvec(s_l) \log\left(1 + \qvec^\xi \cdot \fvec(s_l)\right).
\end{align*}
This yields the final expression for the quenched density:
\begin{equation}
\label{seperation}
h_t^\xi(\sigma) = \exp\left( \boldsymbol{\alpha}_n^\xi \cdot \bar{\fvec}_n(\sigma) +\beta_n^\xi  \right).
\end{equation}
The joint convergence in distribution of $(\boldsymbol{\alpha}_n^\xi, \beta_n^\xi)$ towards $(\mathbf{Z}_s, -\frac{1}{2}\|\mathbf{Z}_s\|^2)$, where $\mathbf{Z}_s \sim \mathcal{N}(\mathbf{0}, sI_{k-1})$, is a direct consequence of a second-order Taylor expansion of the logarithms in their definitions, combined with the orthonormality of the basis $\{f_0,f_1, \dots ,f_{k-1}\}$ and the Central Limit Theorem for Quenched Moments.

 We aim to show that $\rho_{t_n}(\sigma)$, viewed as a random variable on the probability space $(\Omega_n, \pi)$, also converges in distribution to a well-defined limit. This is formally stated in the following proposition.
\begin{proposition}
\label{prop:density_distribution_convergence}
Let $F: \mathbb{R}_+ \to \mathbb{R}$ be any bounded, continuous function. Then,
\[
\lim_{n \to \infty} \int_{\Omega_n} F(\rho_{t_n}(\sigma)) \, \pi(d\sigma) = \mathbb{E}_{\mathbf{Z} \sim \mathcal{N}(\mathbf{0}, I_{k-1})} \left[ F(\psi(\mathbf{Z})) \right],
\]
where the limiting function $\phi: \mathbb{R}^{k-1} \to \mathbb{R}$ is defined as
\[
\psi(\mathbf{u}) := \mathbb{E}_{\mathbf{Z_s} \sim \mathcal{N}(\mathbf{0}, sI_{k-1})} \left[ \exp\left(\mathbf{Z}_s \cdot \mathbf{u} - \frac{1}{2}\|\mathbf{Z}_s\|^2\right) \right].
\]
\end{proposition}
 The proposition states that the random variable $\rho_{t_n}(\sigma)$, where $\sigma \sim \pi$, converges in distribution to the random variable $\psi(\mathbf{Z})$, where $\mathbf{Z} \sim \mathcal{N}(\mathbf{0}, I_{k-1})$. Note that the structure revealed by (\ref{seperation}) is crucial: $\rho_{t_n}(\sigma)$ is a function of the stationary fluctuation vector $\bar{\fvec}_n(\sigma)$ alone. By the Portmanteau Theorem, it is sufficient to prove this convergence for all bounded, Lipschitz continuous functions $F$, which simplifies the subsequent analysis. Assuming this proposition holds, Theorem~\ref{thm:cutoff_profile} follows immediately. 

 \begin{proof}[Proof of Theorem~\ref{thm:cutoff_profile}]
     By applying Proposition~\ref{prop:density_distribution_convergence} to the bounded continuous function $F(x) = 1+x-|x-1|$, and noting that the total mass condition $\int(1+\rho_{t_n})d\pi = 2$ converges to its limit $\mathbb{E}[1+\psi(\mathbf{Z})] = 2$, we obtain the desired convergence of the total variation distance:
\[ 
\lim_{n\to\infty} \|\mu_{t_n} - \pi\|_{\TV} = \|\mathcal{N}(\mathbf{0}, (1+s)I_{k-1}) - \mathcal{N}(\mathbf{0}, I_{k-1})\|_{\TV}. 
\]
 \end{proof}
 
\begin{proof}[Proof of Proposition~\ref{prop:density_distribution_convergence}]
Our proof strategy leverages the fact that $\rho_{t_n}(\sigma)$ is a function of $\bar{\fvec}_n(\sigma)$. The proof thus proceeds in two steps.
First, we show that the sequence of functions $\rho_{t_n}: \mathbb{R}^{k-1} \to \mathbb{R}_+$ converges locally uniformly to a limit function $\psi$ (Lemma~\ref{lem:local_uniform_convergence}).
Subsequently, combining this result with a tail-control argument for the random vector $\bar{\fvec}_n(\sigma)$ (Corollary~\ref{cor:dist_conv_approx_density}), we establish the convergence in distribution of the composite random variable $\rho_{t_n}(\bar{\fvec}_n(\sigma))$ to $\psi(\mathbf{Z})$.
The proofs for the intermediate propositions follow, after which we conclude the main proof.
\end{proof}

\begin{lemma}
\label{lem:local_uniform_convergence}
$\rho_{t_n}: \mathbb{R}^{k-1} \to \mathbb{R}_+$ converges locally uniformly to the limit function $\psi$, i.e., for any compact set $K \subset \mathbb{R}^{k-1}$,
\[
\lim_{n \to \infty} \sup_{\mathbf{u} \in K} |\rho_{t_n}(\mathbf{u}) - \psi(\mathbf{u})| = 0.
\]
\end{lemma}
\begin{proof}
Let the function $g_n$ be defined as
\[
g_n(\mathbf{u},\qvec^\xi) := \exp\left( \boldsymbol{\alpha}_n^\xi \cdot \mathbf{u} +\beta_n^\xi  \right).
\]
We claim that for any $A>0,$
\begin{equation*}
\sup_{\|\mathbf{u}\| \le A} \limsup_{n \ge 1} \max \left\{ \|g_n( \mathbf{u}, \cdot)\|_{\infty}, \|\nabla_{\mathbf{u}} g_n( \mathbf{u}, \cdot )\|_{\infty} \right\} < \infty.
\end{equation*}
Indeed, computing the derivative in $\qvec^\xi$ of $\boldsymbol{\alpha}_n^\xi \cdot \mathbf{u} +\beta_n^\xi$ shows that it is maximized at $\qvec^\xi=\mathbf{u}/\sqrt{n}$ and from this, we deduce the bound on $\|g_n( \mathbf{u}, \cdot)\|_{\infty}.$ Regarding the gradient, observe that $\|\nabla_{\mathbf{u}} g_n( \mathbf{u},\qvec^\xi)\|_\infty \le \sup_{\|e\|_{\infty}\le 1} \|g_n(u+e,\qvec^\xi)\|_\infty.$
As a consequence of the claim, we deduce that for $\|u\|\le A$, the map $g_n(u, \cdot)$ coincides with $g_n(u, \cdot) \wedge M$ for some positive constant $M$. The latter is a continuous bounded function of $\boldsymbol{\alpha}_n^\xi$ and $\beta_n^\xi$, so the convergence in law implies that
\begin{equation*} 
\lim_{n\to\infty} \mathbb{E}_\xi[g_n(u,\qvec^\xi)] = \mathbb{E}_{\mathbf{Z}_s}[e^ {\mathbf{Z}_s \cdot \mathbf{u} -\frac{1}{2}\|\mathbf{Z}_s\|^2}].
\end{equation*}
The bound on the derivative in $u$ of $g_n(u, \qvec^\xi)$ proven above suffices to deduce that $u \mapsto  \mathbb{E}_\xi[g_n(u,\qvec^\xi)]$ is equicontinuous on $\|u\|\le A$. The Arzelà-Ascoli theorem guarantees that the convergence is uniform.

\end{proof}
\begin{corollary}
\label{cor:dist_conv_approx_density}
Let $F: \mathbb{R}_+ \to \mathbb{R}$ be a bounded, Lipschitz continuous function. Then,
\[
\lim_{n\to\infty} \int \left| F(\rho_{t_n}(\bar{\mathbf{f}}_n(\sigma))) - F(\psi(\bar{\mathbf{f}}_n(\sigma))) \right| \, \pi(d\sigma) = 0.
\]
\end{corollary}
With Corollary~\ref{cor:dist_conv_approx_density} at our disposal, we obtain the distributional convergence \\$\rho_{t_n}(\bar{\mathbf{f}}_n(\sigma))  \xrightarrow{d} \psi(\mathbf{Z})$. This is because the corollary implies that it suffices to determine the limit of $\mathbb{E}[F(\psi(\bar{\mathbf{f}}_n(\sigma)))]$. Since the composite function $F \circ \psi$ is bounded and continuous, and $\bar{\mathbf{f}}_n(\sigma) \xrightarrow{d} \mathbf{Z} \sim \mathcal{N}(\mathbf{0}, I_{k-1})$ by the Central Limit Theorem, the Continuous Mapping Theorem yields the required limit:
\[
\lim_{n\to\infty} \E[F(\psi(\bar{\mathbf{f}}_n(\sigma)))] = \E[F(\psi(\mathbf{Z}))].
\]

\begin{proof}[Proof of Corollary~\ref{cor:dist_conv_approx_density}]
It suffices to show that for any given $A>0$
\begin{align*}
    \limsup_{n \rightarrow \infty} \int \left| F(\rho_{t_n}(\bar{\mathbf{f}}_n(\sigma))) - F(\psi(\bar{\mathbf{f}}_n(\sigma))) \right| \mathbf{1}_{\{\|\bar{\fvec}_n(\sigma)\| \le A\}} \, \pi(d\sigma)=0,
\end{align*}
and that 
\begin{align*}
    \lim_{A \rightarrow \infty}\limsup_{n \rightarrow \infty} \int \left| F(\rho_{t_n}(\bar{\mathbf{f}}_n(\sigma))) - F(\psi(\bar{\mathbf{f}}_n(\sigma))) \right| \mathbf{1}_{\{\|\bar{\fvec}_n(\sigma)\| > A\}} \, \pi(d\sigma)=0.
\end{align*}
 The first limit is zero due to the Lipschitz continuity of $F$ and the local uniform convergence of $\rho_{t_n}$ established in Lemma~\ref{lem:local_uniform_convergence}. For the second limit, the integrand is uniformly bounded since $F$ is bounded, while the measure of the domain, $\pi(\|\bar{\fvec}_n(\sigma)\| > A)$, is bounded by a term of the form $C\exp(-c A^2)$ via Hoeffding's inequality.
\end{proof}

\appendix
\section{Proof of Lemmas \ref{lem:graphical_construction} and \ref{lem:stationary_measure}}

In this appendix, we provide the proof of Lemmas \ref{lem:graphical_construction} and \ref{lem:stationary_measure} regarding the ergodicity of the system for the self-containedness of the article. 
 
 \begin{proof}[Proof of Lemma \ref{lem:graphical_construction}]
    
    We start by observing that for every $t \in \mathbb{N}$, the distribution $\mu_t$ is given by the average over all partitions of $[n]$ into $N=2^t$ disjoint sets (which are allowed to be empty):
    \[
    \mu_t = \frac{1}{N^n} \sum_{(A_1, \dots, A_N)} \mu_{A_1} \otimes \mu_{A_2} \otimes \dots \otimes \mu_{A_N},
    \]
    where the sum is taken over all $N^n$ such partitions. The case $t=1$ is the definition of uniform recombination. The case $t \ge 2$ follows by induction, since if $(A_1, \dots, A_N)$ and $(A'_1, \dots, A'_N)$ are two independent uniformly random partitions of $[n]$ into $N$ disjoint sets, and $A \subset [n]$ is an independent and uniformly random subset, then
    \[
    (A_1 \cap A, \dots, A_N \cap A, A'_1 \cap A^c, \dots, A'_N \cap A^c)
    \]
    yields a uniformly random partition of $[n]$ into $2N$ disjoint sets.

    A way to sample such a partition uniformly at random is to assign to each coordinate $i \in [n]$ an independent and uniformly random number $U_i \in \{1, \dots, N\}$ and define the sets as $A_x := \{i \in [n] : U_i = x\}$ for $x = 1, \dots, N$. With this definition, $\mu_t$ becomes an expectation:
    \[
    \mu_t = \mathbb{E} \left[ \mu_{A_1} \otimes \mu_{A_2} \otimes \dots \otimes \mu_{A_N} \right].
    \]
    Let $\xi(x)$ for $x \in \{1, \dots, N\}$ be $N$ independent random configurations, each with law $\mu$. For any configuration $\sigma \in \Omega_n$, we can compute its probability using Fubini's theorem:
    \begin{align*}
\mu_t(\sigma) &= \mathbb{E}_{(U_i)} \left[ \prod_{x=1}^N \mu_{A_x}(\sigma) \right] \\
&= \mathbb{E}_{(U_i)} \left[ \mathbb{P}_{\xi} \left( \forall x \in \{1, \dots, N\}, \forall i \in A_x, \xi_i(x) = \sigma_i \right) \right] \\
&= \mathbb{E}_{\xi} \left[ \mathbb{P}_{(U_i)} \left( \forall x \in \{1, \dots, N\}, \forall i \in A_x, \xi_i(x) = \sigma_i \right) \right] \\
&= \mathbb{E}_{\xi} \left[ \mathbb{P}_{(U_i)} \left( \forall i \in [n], \xi_i(U_i) = \sigma_i \right) \right].
\end{align*}
    The right-hand side is just the law of the random configuration $\sigma^* = (\xi_1(U_1), \dots, \xi_n(U_n))$ by definition. This completes the proof.
    \end{proof}

    \begin{proof}[Proof of Lemma \ref{lem:stationary_measure}]
We first show that $\pi$ is stationary under the recombination operator $\circ$. By definition of the operator:
\[
    \pi \circ \pi = 2^{-n} \sum_{A \subseteq [n]} (\pi)_A \otimes (\pi)_{A^c}.
\]
Since $\pi$ is a product measure, its marginal on a coordinate subset $A$ is simply the product of the single-site marginals over that subset:
\[
    (\pi)_A = \bigotimes_{i \in A} p_i \quad \text{and} \quad (\pi)_{A^c} = \bigotimes_{j \in A^c} p_j.
\]
The product of these two marginals reconstructs the original product measure over all coordinates:
\[
    (\pi)_A \otimes (\pi)_{A^c} = \left( \bigotimes_{i \in A} p_i \right) \otimes \left( \bigotimes_{j \in A^c} p_j \right) = \bigotimes_{k \in [n]} p_k = \pi.
\]
Since every term in the summation is equal to $\pi$, and there are $2^n$ subsets $A \subseteq [n]$, the expression becomes:
\[
    \pi \circ \pi = 2^{-n} \sum_{A \subseteq [n]} \pi = 2^{-n} (2^n \pi) = \pi.
\]
Thus, $\pi$ is a fixed point of the dynamics.

To prove that $\mu_t \to \pi$ as $t \to \infty$, we explicitly define the underlying fragmentation process. Let $(\mathcal{A}_t)_{t \ge 0}$ be a sequence of random partitions of $[n]$.
    \begin{itemize}
        \item At time $t=0$, the partition is trivial: $\mathcal{A}_0 = \{[n]\}$.
        \item For $t \ge 0$, the partition $\mathcal{A}_{t+1}$ is obtained from $\mathcal{A}_t$ as follows: Let $A \subseteq [n]$ be a new, independent, uniformly random subset. Every block $B \in \mathcal{A}_t$ is split into two new blocks, $B \cap A$ and $B \cap A^c$. Thus, $\mathcal{A}_{t+1} = \{ B \cap A, B \cap A^c \mid B \in \mathcal{A}_t \}$.
    \end{itemize}
    At each time $t$, this process yields a uniformly random partition of $[n]$ into $2^t$ subsets (some of which may be empty). This process is equivalent to the graphical construction where each coordinate $i \in [n]$ is assigned an independent, uniformly random leaf $U_i \in \{1, \dots, 2^t\}$. The resulting partition is $\mathcal{A}_t = \{A_x\}_{x=1}^{2^t}$ where $A_x = \{i \in [n] : U_i = x\}$.
    
    Define the \textit{fragmentation time} $\tau_{\text{frag}}$ as the first time $t$ at which $\mathcal{A}_t$ separates all pairs of coordinates. That is,
    \[
    \tau_{\text{frag}} = \min\{t \ge 0 \mid \forall i \neq j \in [n], \exists B_1, B_2 \in \mathcal{A}_t, B_1 \neq B_2 \text{ s.t. } i \in B_1, j \in B_2 \}.
    \]
    Conditionally on the event $\{\tau_{\text{frag}} \le t\}$, the configuration $\sigma^*$ at the root of the tree has the law $\pi$. This is because if all coordinates are in separate blocks of the partition, they are assigned spins from independently chosen leaves, making the resulting spins $\sigma_i^*$ mutually independent. The law of $\sigma^*$ is thus the product of its marginals, which is $\pi$.
    
    The total variation distance can be bounded by the probability of the complementary event:
    \[
    \|\mu_t - \pi\|_{\TV} \le \p(\tau_{\text{frag}} > t).
    \]
    The event $\{\tau_{\text{frag}} > t\}$ occurs if there exists at least one pair of coordinates $(i, j)$ with $i \neq j$ that has not been separated by time $t$. At any single step of the process, a pair $(i, j)$ is separated if they fall into different parts of the random bipartition defined by $A$ and $A^c$, an event of probability $1/2$. The probability they remain unseparated after one step is $1/2$. After $t$ independent steps, the probability that they are still in the same block of the partition $\mathcal{A}_t$ is $(1/2)^t = 2^{-t}$.
    By the union bound over all $\binom{n}{2}$ pairs of coordinates:
    \[
    \p(\tau_{\text{frag}} > t) \le \sum_{1 \le i < j \le n} \p(\{i, j\} \text{ are not separated by time } t) = \binom{n}{2} 2^{-t}.
    \]
    As $t \to \infty$, this upper bound converges to 0. Therefore, $\|\mu_t - \pi\|_{\TV} \to 0$, which proves convergence. 
    
    The uniqueness is a direct consequence of the convergence itself; for if another stationary measure $\pi_s \neq \pi$ existed, the evolution starting from $\mu_0 = \pi_s$ would remain constant at $\pi_s$, contradicting the proven result that every evolution converges to the single limit $\pi$.
    
    \end{proof}

\noindent\textbf{Acknowledgment.} I. Seo and J. Kim were supported by the National Research Foundation of Korea (NRF) grant funded by the Korean government (MEST) No. RS-2023-NR076621 and RS-2025-23525546. 

\noindent\textbf{Conflict of Interest}
The authors declare no competing interests.

\noindent\textbf{Data Availability Statement}
Data sharing is not applicable to this article as no datasets were generated or analysed during the current study.

\bibliographystyle{plain}  
\bibliography{ref}

\end{document}